\documentclass[10pt,twocolumn,twoside]{IEEEtran}

\usepackage{multicol}
\usepackage{graphicx,float,subfigure}
\usepackage{url}
\usepackage{amssymb,amsmath,amsfonts,layout}
\usepackage{makeidx}
\usepackage{blkarray,multirow}
\usepackage{cite}
\usepackage{enumerate}
\usepackage{algpseudocode}
\usepackage{algorithm}

\usepackage{amsthm}
\usepackage{tikz}
\usetikzlibrary {positioning}

\algnewcommand{\Inputs}[1]{%
	\State \textbf{Inputs:}
	\Statex \hspace*{\algorithmicindent}\parbox[t]{.8\linewidth}{\raggedright #1}
}
\algnewcommand{\Initialize}[1]{%
	\State \textbf{Initialize:}
	\Statex \hspace*{\algorithmicindent}\parbox[t]{.8\linewidth}{\raggedright #1}
}
\algnewcommand{\Outputs}[1]{%
	\State \textbf{Outputs:}
	\Statex \hspace*{\algorithmicindent}\parbox[t]{.8\linewidth}{\raggedright #1}
}

\usepackage{tikz}
\usetikzlibrary{shapes,arrows,positioning} 
\tikzset{
	>=stealth',
	block/.style={
		rectangle,
		rounded corners,
		draw=black, very thick,
		text width=12em,
		minimum height=3em,
		text centered},
	link/.style={
		->,
		thick,
		shorten <=2pt,
		shorten >=2pt},
	decision/.style={
		diamond,
		draw, very thick,
		fill=blue!20, 
		text width=8em,
		aspect=3,
		text centered}
}

\theoremstyle{plain}

\newtheorem{assumption}{Assumption}
\newtheorem{theorem}{Theorem}[section]
\newtheorem{proposition}[theorem]{Proposition}
\newtheorem{lemma}[theorem]{Lemma}

\newtheorem{remark}[theorem]{Remark}

\newcommand{\longthmtitle}[1]{\mbox{}{\bf \textit{(#1).}}}

\newcommand{\oprocendsymbol}{\hbox{$\square$}}
\newcommand{\oprocend}{\relax\ifmmode\else\unskip\hfill\fi\oprocendsymbol}

\makeatletter
\newcommand{\pushright}[1]{\ifmeasuring@#1\else\omit\hfill$\displaystyle#1$\fi\ignorespaces}
\newcommand{\pushleft}[1]{\ifmeasuring@#1\else\omit$\displaystyle#1$\hfill\fi\ignorespaces}
\renewcommand*\env@matrix[1][*\c@MaxMatrixCols c]{%
	\hskip -\arraycolsep
	\let\@ifnextchar\new@ifnextchar
	\array{#1}}
\makeatother

\DeclareMathOperator{\tr}{\bf Tr}
\DeclareMathOperator{\W}{\mathcal{W}}
\DeclareMathOperator{\N}{\mathcal{N}}

\DeclareMathOperator{\Pint}{\mathbb{N}}

\DeclareMathOperator{\sink}{\mathbb{K}}
\DeclareMathOperator{\R}{\mathbb{R}}

\DeclareMathOperator{\E}{\mathcal{E}}
\DeclareMathOperator{\X}{\mathcal{X}}
\DeclareMathOperator{\I}{\mathcal{I}}
\DeclareMathOperator{\St}{\mathcal{S}}
\DeclareMathOperator{\graph}{\mathcal{G}}
\newcommand{\Hc}{\ensuremath{\mathcal{H}}}
\newcommand{\Nc}{\ensuremath{\mathcal{N}}}
\newcommand{\argmin}{\operatorname{argmin}}

\newcommand{\rank}{\operatorname{rank}}

\newcommand{\real}{\ensuremath{\mathbb{R}}}
\newcommand{\complex}{\ensuremath{\mathbb{C}}}

\renewcommand{\bar}{\overline}

\newcommand{\setdef}[2]{\{#1 \; | \; #2\}}

\newcommand{\map}[3]{#1:#2 \rightarrow #3}

\allowdisplaybreaks

\renewcommand{\baselinestretch}{1}

\title{Scheduled-Asynchronous Distributed Algorithm
  \\
  for Optimal Power Flow}

\author{Chin-Yao Chang \quad Jorge Cort\'es \quad Sonia Mart{\'\i}nez
  \thanks{C.-Y. Chang, Jorge Cort\'es, and Sonia Mart{\'\i}nez are
    with the Department of Mechanical and Aerospace Engineering,
    University of California, San Diego, CA, USA. Email: {\footnotesize {\tt
        \{chc433,cortes,soniamd\}@ucsd.edu}}} \thanks{A preliminary
    version of this work appeared as~\cite{CYC-JC-SM:17-acc} at the
    2017 American Control Conference.}}
      
\begin{document}
\maketitle

\begin{abstract}
  Optimal power flow (OPF) problems are non-convex and large-scale
  optimization problems with important applications in power networks.
  This paper proposes the scheduled-asynchronous algorithm to solve a
  distributed semidefinite programming (SDP) formulation of the OPF
  problem. In this formulation, every agent seeks to solve a local
  optimization with its own cost function, physical constraints on its
  nodal power injection, voltage, and power flow of the lines it is
  connected to, and decision constraints on variables shared with
  neighbors to ensure consistency of the obtained solution.  In the
  scheduled-asynchronous algorithm, every pair of connected nodes in
  the electrical network update their local variables in an
  alternating fashion. This strategy is asynchronous, in the sense
  that no clock synchronization is required, and relies on an
  orientation of the electrical network that prescribes the precise
  ordering of node updates.  We establish the asymptotic convergence
  properties to the primal-dual optimizer when the orientation is
  acyclic.  Given the dependence of the convergence rate on the
  network orientation, we also develop a distributed graph coloring
  algorithm that finds an orientation with diameter at most five for
  electrical networks with geometric degree distribution.
  Simulations illustrate our results on various IEEE bus test cases.
\end{abstract}

\section{Introduction}\label{sec: Intro}
% Finding an optimal solution for the OPF problem in a distributed way
% is non-trivial but important in various applications. Recent studies
% show that a wide class of OPF problem has an exact semidefinite
% programming (SDP) convex relaxation. Solving the SDP associated with
% the OPF problem in a distributed way is promising, while only few
% works have considered such extension.
The optimal power flow (OPF) problem seeks to minimize the cost of
electricity generation subject to voltage and power flow
constraints. Finding a solution to the OPF problem is challenging due
to its large-scale and non-convex nature and, therefore, it is
typically solved off-line for centralized planning of power
networks. However, recent technological advances involving the
integration of renewable distributed energy resources introduce higher
operational uncertainty in managing the electrical grid, motivating
the need for real-time methods to solve OPF
problems. %~\cite{AL-BZ-ND:12}
Ideally, such methods should enjoy robustness against disturbances and
tolerance against intermittent engagement of energy resources. An
additional consideration further justifying the need for such methods
is the increase of plug-and-play devices in distribution networks and
the ensuing uncertain overall system configuration. Motivated by these
considerations, this paper introduces a provably-correct distributed
algorithm to solve a convexified OPF problem over a power network.

\textit{Literature review:} Finding a global optimum of the OPF
problem is challenging due to its non-convexity, so most existing
algorithms only guarantee a local optimum, see
e.g.,~\cite{JM-ME-RA:99,VA-CC:92,MA:02}. A commonly used approach in
the literature~\cite{AJC-JAA:98,PNB-AGB-NIM-NKP:05} to convexify the
problem is the use of DC power flow equations. An alternative route
for the convexification of the OPF problem is employing semi-definite
programming (SDP).  The work~\cite{JL-SL:12} shows that the SDP convex
relaxation on the OPF problem is exact for many networks, a fact that
allows to find a global solution. Conditions on the exact convex
relaxation have been further established
in~\cite{BZ:13,DKM:13}. Several recent
works~\cite{DKM-CJ-IH-PP:17,RM-MA-JL:14} further develop SDP-based
algorithms for a near-global optimal solution for OPF problems where
the SDP convex relaxation does not provide a feasible solution.
Relatively few works consider solving the OPF problem in a distributed
way.  The paper~\cite{PB-AB-NM-NP:05} proposes a distributed approach
to the DC-OPF problem, where the electrical network is decomposed into
several regions and each region solves its regional DC-OPF problem
while iteratively matching its tie-line powers with the connected
regions.  The work~\cite{AL-BZ-ND:12} considers an SDP-based
convexification of the OPF problem and then proposes gradient-based
primal-dual algorithm which displays fast convergence to the optimum
in simulation studies.  The algorithm design is limited to voltage
magnitude constraints and linear objective functions and does not
incorporate constraints on the active/reactive power. The
work~\cite{ED-HZ-GBG:13} also considers a SDP-based relaxation of the
OPF problem, partitions a so-called weakly-meshed network into several
areas so that the ``macro'' graph describing the interconnected areas
has a tree topology, and applies the alternating direction method of
multipliers (ADMM) to solve the distributed optimization associated
with the macro graph.  In general, gradient-based methods are amenable
to asynchronous implementations, but have a rate of convergence
$O(\log(n)/n)$ or slower~\cite{DJ-JX-JMM:14}. In comparison, ADMM is
faster with a $O(1/n)$ convergence rate~\cite{BH-XY:15}, while
requiring a synchronous
implementation.  % This motivates the design of distributed
% algorithms with fast convergence properties that at the same and
% asynchronous property is desired.
Drawing connections with the increasing body of work on gossiping in
network systems~\cite{SL-AN:16,LM-AS-FT:13,DS:09}, our algorithm
design prescribes pairwise updates between neighboring agents to avoid
the requirement of clock synchronization while enjoying similar
convergence properties as ADMM.
			
\textit{Statement of contributions:} Our starting point for the
algorithm design is the formulation of a distributed version of the
OPF problem, where each bus (node) plays the role of an agent with
computation and communication capabilities. Every node creates local
copies of the voltage of the neighboring nodes which are physically
connected to it. Each node poses an optimization with its local cost
function, constraints of its nodal power injection and voltage, and
constraints associated with the power flow of the connected lines. On
top of those local physical constraints, every node also has equality
constraints sharing to its neighbors to ensure that every copy of the
variables coincides.  We then consider the convexification of this
distributed OPF formulation using semi-definite programming (SDP).
Our first contribution is the synthesis of the scheduled-asynchronous
algorithm to solve the distributed OPF formulation in a distributed
way.  In our design, every pair of connected buses solves their local
optimization problem in an ADMM-like, alternating fashion.  To
establish the order of such alternating iterations across the network,
every bus only updates its variables when all the neighboring buses
have finished an update later than its last update. This logic does
not require any clock synchronization between the agents, in contrast
to ADMM.  Our second contribution is the convergence analysis of the
proposed scheduled-asynchronous algorithm. % The algorithm is
% robust to mild network effect of packet drop.  We demonstrate the
% robustness through both analytical analysis and simulations.
Under reasonable assumptions on the data defining the optimization
problem and the requirement than the graph orientation is acyclic (to
avoid ``locked'' situations where every bus is waiting for an update
from at least one of its neighbors), we employ the LaSalle Invariance
Principle to establish the convergence to the primal-dual solution.
% If the network is bipartite, the algorithm convergence rate is
% $O(1/n)$. One important special case of the bipartite graph is tree
% network topology. Considering that tree network is commonly seen in
% the distribution networks, the proposed algorithm is especially
% suitable for the OPF problem in distribution networks.
Given the dependence of the algorithm convergence on the orientation
of the network graph, our third contribution concerns the optimal
selection of this orientation. In fact, the time need for all agents
to finish at least one update is directly related to the diameter of
the oriented network graph, and hence it is desirable to minimize it.
In general, finding an orientation that minimizes the diameter is
equivalent to finding the chromatic number of the graph, which is
NP-hard. We design a distributed graph coloring algorithm that finds
an orientation with diameter at most five for electrical networks with
geometric degree distribution. The distributed nature of this strategy
makes it naturally robust to changes in the network
topology. % We can implement the graph
% coloring algorithm with the scheduled-asynchronous algorithm in a
% separate time scale because the network topology changes slowly
% compared to other uncertainties of the OPF problem such as the change
% of loads.
Simulations on various IEEE bus test cases of the combined graph
coloring and scheduled-asynchronous algorithms illustrate the
performance and robustness of the proposed design.
		
\textit{Organization:} Section~\ref{sec:preliminaries} presents basic
concepts and notation.  Section~\ref{sec: ProbSet} introduces the OPF
problem and its distributed formulation. Section~\ref{sec:DisAlgo}
presents the scheduled-asynchronous algorithm and analyzes its
convergence properties. Section~\ref{sec:Direct_graph} introduces a
distributed graph coloring algorithm to obtain an acyclic orientation
of small diameter.  Section~\ref{sec:Simulation} illustrates the
effectiveness of the proposed algorithm in benchmark IEEE test
cases. Finally, we gather our conclusions in
Section~\ref{sec:conclusion}.
	
\section{Preliminaries}\label{sec:preliminaries}
This section introduces basic notation and concepts from graph theory
and optimization.
	
\subsection{Notation}
We denote by $\Pint$, $\real$ and $\complex$ the sets of positive
integers, reals and complex numbers, respectively.  We denote by
$|\N|$ the cardinality of the set~$\N$.  For a complex number
$a\in\mathbb{C}$, we let $|a|$ and $\angle a$ be the complex modulus
and angle of $a$. The $2$-norm of a complex vector $v \in \complex^n$
is written as $\|v\|$. Let $\mathbb{S}_+ \subset \complex^{n \times
  n}$ and $\mathcal{H}^n \subset \mathbb{S}_+$ be the set of positive
semidefinite and $n$-dimensional Hermitian matrices, respectively. For
$A\in\mathbb{C}^{n\times n}$, we let $A^*$ be its conjugate transpose
and $\tr\{A\}$ be its trace.  For $A,B\in\mathcal{H}^n$, we denote
their inner product by $\langle A,B \rangle=\tr\{AB\}$.  We use
${\bigtriangledown} F$ to denote the gradient of the scalar
function~$F$.
	
\subsection{Graph Theory}\label{sec:graph}
We review basic notions of graph theory
following~\cite{FB-JC-SM:08cor}.  A graph is a pair $\graph =
(\N,\E)$, where $\N \subseteq\Pint$ is its set of vertices or nodes
and $\E\subseteq\N\times\N$ is its set of edges. A \textit{loop} is an
edge that connects a vertex to itself. Two nodes $i,k\in\N$ are
\textit{connected} if $\{i,k\}\in\E$. The graph is \textit{undirected}
if $\{i,k\} = \{k,i\} \in\E$. The \textit{local neighborhood} of a
node $k$ in an undirected graph is $\N_k: =\setdef{l \in \N}{
  \{l,k\}\in\E}\cup\{k\}$. The \textit{degree} of a node $k$ is
$|\N_k|-1$.  In a directed graph, each pair of vertices in $\E$ is
ordered such that the corresponding edge $\{i,k\}$ has a direction,
with $i$ and $k$ distinguished as the \textit{tail} and \textit{head}
nodes, respectively. Node $i$ is a \textit{source} (resp.
\textit{sink}) if it is the tail (resp. head) node of all the edges it
belongs to.  Node $k$ is an \textit{out-neighbor} of node $i$ if
$\{i,k\}\in\E$. The \textit{out-degree} of node $i$ is defined as
$|\{k \in \N|\{i,k\}\in\E\}|$. A \textit{path} in a (directed or
undirected) graph is a sequence of vertices such that any two
consecutive nodes correspond to an edge of the graph. The
\textit{length} of a path is the number of its corresponding
edges. The \textit{diameter} of a graph is the maximum length of the
shortest path connecting any two graph vertices in the graph. A
\textit{cycle} is a path whose first and last vertices are the same.
A graph is \textit{acyclic} if it contains no cycles.  An
\textit{orientation} of an undirected graph is an assignment of
exactly one direction to each of its edges.  A graph orientation is
acyclic if the resulting directed graph is acyclic. A graph is
\textit{bipartite} if the set of its vertices can be decomposed into
two disjoint subsets such that, within each one, no two vertices are
connected.

A \textit{simple graph} is a graph with no loops nor multiple edges
connecting any pair of two vertices.  A \textit{planar graph} is a
graph that can be drawn on the plane in a way that its edges intersect
only at their endpoints.  A \textit{vertex-induced subgraph} of
$\graph=(\N,\E)$, written as $\graph_s[\N_s]$, is a subgraph of
$\graph$ with the set of nodes $\N_s\subseteq\N$ and set of edges
$\E_s = \E\cap(\N_s\times\N_s)$. A \textit{chordal graph} is a graph
that does not contain an induced cycle of length greater than four.
\textit{Graph coloring} consists of assigning a color to every node in
the graph in such a way that any pair of connected nodes have
different colors. The smallest number of colors needed to color a
graph $G$ is called its chromatic number.

\subsection{Strong Duality of Convex Optimization}
Here, we review some fundamental concepts in convex optimization
following~\cite{SB-LV:09}. Consider a convex optimization problem of
the form
\begin{align}
  \label{eq:convex_opt}
  \min_{x} f_0(x), \quad\text{s.t. } Ax = b,\; f_i(x)\leq 0, \;i =
  1,\dots,m,
\end{align}
where $f_0,\dots,f_m:\real^n\rightarrow\real$ are convex functions,
$A\in\real^{n\times r}$, $b\in\real^r$, and $Ax = b$ defines affine
equality constraints. The dual problem of
optimization~\eqref{eq:convex_opt} is given as
\begin{align}
  \label{eq:convex_opt_dual}
  \max_{\lambda\geq 0,\mu}\Big(\min_x f_0(x) + \sum_{i=1}^m \lambda_i
  f_i(x) + \mu^\top(Ax-b) \Big),
\end{align}
where $\lambda\in\real^m$ and $\mu\in\real^r$ are known as Lagrange
multipliers. Let $p^\star$ and $d^\star$ be the optimal value of the
primal and dual problems, respectively.  Strong duality holds if
$p^\star = d^\star$.  Under strong duality, the Karush-Kuhn-Tucker
(KKT) conditions are a necessary and sufficient characterization of
the optimality of the primal-dual solution
$(x^\star,\lambda^\star,\mu^\star)$,
\begin{align*}
  \begin{cases}
    0 \in \bigtriangledown f_0(x^\star) +
    \sum_{i=1}^m\lambda_i^\star\hspace{-1mm} \bigtriangledown
    f_i(x^\star) + (\mu^\star)^\top Ax^\star,
    \\
    \lambda_i^\star f_i(x^\star) = 0, \quad \forall i = 1,\dots,m,
    \\
    (\mu^\star)^\top (Ax^\star-b ) = 0,
    \\
    Ax^\star=b,\quad f_i(x^\star)\leq 0, \quad \forall i = 1,\dots,m,
    \\
    \lambda_i^\star\geq 0, \quad \forall i = 1,\dots,m.
  \end{cases}
\end{align*}
These conditions correspond to stationarity, complementary slackness,
and primal and dual feasibility, respectively.  The (refined) Slater's
condition holds if there exists $x \in \real^n$ with
\begin{align*}
  Ax = b \text{ and } f_i(x)<0, \quad \forall i=1,\dots,m .
\end{align*}
Slater's condition implies that strong duality holds.

\section{Problem Formulation}\label{sec: ProbSet}
This section introduces the problem of interest. We begin with a
general formulation of the optimal power flow (OPF) problem over an
electrical network. We then consider its convex relaxation and rewrite
it as the combination of several smaller-scale interconnected convex
subproblems. The resulting semidefinite programming (SDP) problem is
the starting point for our distributed algorithm design.
	
Consider an electrical network graph with generation buses $\N_G$,
load buses $\N_L$, and electrical interconnections described by an
undirected edge set $\E$.  Let $\N= \N_G \cup \N_L$ and denote its
cardinality by~$N$.  We denote the phasor voltage at bus $i$ by $V_i =
E_i e^{j \theta_i}$, where $E_i\in\mathbb{R}$ and
$\theta_i\in[-\pi,\pi)$ are the voltage magnitude and phase angle,
respectively.  When convenient, we let $V = \setdef{V_i}{i\in\N}$
denote the collection of voltages at all buses. The active and
reactive power injections at bus $i$ are given by the power flow
equations~\cite{SB-DG-SL-KC:11}
\begin{align*}
  P_i & = \tr\{Y_iVV^*\} + P_{D_i},
  \\
  Q_i& = \tr\{\bar{Y}_iVV^*\} + Q_{D_i},
\end{align*}
where $P_{D_i},Q_{D_i}\in\R$ are the active and reactive power
demands\footnote{Some buses may have generation and load
  simultaneously. For buses with only generators, $P_{D_i},Q_{D_i}$
  are both zero.} at bus $i$, and $Y_i, \bar{Y}_i \in\mathcal{H}^{N}$
are derived from the admittance matrix
\textbf{Y}$\in \complex^{N \times N}$ as follows
\begin{subequations}\label{eq:admit}
  \begin{align}
    Y_i & = \frac{(e_ie_i^\top\textbf{Y})^* + e_ie_i^\top\textbf{Y}}{2},
    \\
    \bar{Y}_i & = \frac{(e_ie_i^\top\textbf{Y})^* -
      e_ie_i^\top\textbf{Y}}{2j}.
  \end{align}  
\end{subequations}
Here $\{e_i\}_{i= 1,\dots,N}$ denotes the canonical basis of
$\real^N$.  The OPF problem also involves the following box
constraints
\begin{align}\label{eq:allConst}
  &\underline{V}_i^2 \leq |V_i|^2 \leq \bar{V}_i^2, \; \forall i\in
  \mathcal{N}, \nonumber
  \\
  &\underline{P}_i \leq P_i \leq \bar{P}_i, \;\; \underline{Q}_i \leq
  Q_i \leq \bar{Q}_i, \; \forall i\in
  \mathcal{N},
  \\
  \nonumber &|V_{i} - V_{k}|^2 \leq \bar{V}_{ik}, \;\forall \{i,k\}
  \in \mathcal{E},
\end{align}
where $\bar{V}_{ik}$ is the upper bound of the voltage difference
between buses $i,k$, and $\underline{V}_i$ and $\bar{V}_i$ are the
lower and upper bounds of the voltage magnitude at bus $i$,
respectively. The quantities $\underline{P}_i,\underline{Q}_i,
\bar{P}_i, \bar{Q}_i$, are defined similarly. The objective function
for the OPF problem is typically given as a quadratic function of the
active power injection,
\begin{align}\label{eq:Cost}
  & \sum_{k \in \mathcal{N}_G}c_{i2} P_i^2 + c_{i1} P_i,
\end{align}
where $c_{i2}\geq 0$, and $c_{i1}\in \mathbb{R}$. Using
$W = VV^*\in \mathcal{H}^N $ as the decision variable, the
OPF problem is formulated as follows
\begin{flalign*}
  % \label{eq:Opt1}
  \textbf{(P1)} \hspace{1mm}\min_{W}\hspace{-1mm} \sum_{i \in
    \N_G}\hspace{-1mm} c_{i2}
  (\tr\{Y_iW\}+\hspace{-1mm}P_{D_i})^2\hspace{-2mm} +\hspace{-1mm}
  c_{i1} (\tr\{Y_iW\}\hspace{-1mm}+P_{D_i}), &&
\end{flalign*}
subject to
\begin{subequations}\label{eq:const}
  \begin{alignat}{5}
    & \underline{P}_i \leq \tr\{Y_iW\} + P_{D_i} \leq \bar{P}_i, \;
    \forall i\in \mathcal{N},
    \label{eq:const-c}
    \\
    & \underline{Q}_i \leq \tr\{\bar{Y}_iW\} + Q_{D_i} \leq \bar{Q}_i
    , \; \forall i\in \mathcal{N} ,
    \label{eq:const-d}
    \\
    & \underline{V}_i^2 \leq \tr\{M_iW\} \leq \bar{V}_i^2, \; \forall
    i\in \mathcal{N} , \displaybreak[0]
    \label{eq:const-e}
    \\
    & \tr\{M_{ik}W\} \leq \bar{V}_{ik}, \; \forall \{i,k\} \in
    \mathcal{E},
    \label{eq:const-f}
    \\
    & W \succeq 0, \quad\rank(W)=1,
    \label{eq:const-g}
  \end{alignat}
\end{subequations}
where $M_i, M_{ik}\in \mathcal{H}^{N}$ are defined so that
$\tr\{M_iW\} = |V_i|^2$ and $\tr\{M_{ik}W\} = |V_i-V_k|^2$.
Constraints~(\ref{eq:const-c}-\ref{eq:const-f}) come from
Eq.~\eqref{eq:allConst}. The combined constraints $W \succeq 0$ and
$\rank(W)=1$ in~\eqref{eq:const-g} correspond to writing the
voltage as a matrix variable. The elimination of the rank constraint
gives rise to the convex relaxation of the OPF problem. 
	
Following the exposition in~\cite{CYC-WZ:16}, we next reformulate the
OPF problem in a distributed way as follows. We start from the
observation that every constraint except Eq.~\eqref{eq:const-g} in
\textbf{(P1)} is either related to the power injection at one bus or
the voltage difference between connected buses. In addition, the power
injection at one bus is only related to the voltage of the buses it is
connected to. This property is embedded in the structure of the
non-zero entries of the admittance matrix in~\eqref{eq:admit}. As a
consequence, we can rewrite the
constraints~(\ref{eq:const-c}-\ref{eq:const-f}) in terms of variables
$W_i \in \mathcal{H}^{N_i}$ for all $i\in\N\textbf{}$, where $W_i$ is
quadratic in the voltage variables corresponding to the local
neighborhood $\N_i$ in the electrical network graph, and $N_i =
|\N_i|$.
% \margins{I don't think we have given a formal definition of $W_i$,
%   haven't we? I understand that $W_i$ is obtained similarly to $W$
%   but only considering the voltages $\{V_j\,|\,j \in \Nc_i \cup
%   \{i\}\}$, is this correct? And so, what does the notation for
%   $W_{\Nc}$ mean?  The notation that we are using for vectors is
%   given in terms of $\{\dots \}$, see for example the notation given
%   for $V$. But if we interpret thenotation for $W_{\Nc}$ mean? if we
%   interpret it as in a vector notation, with matrices $W_i$ in place
%   of real numbers, it does not make any sense, does it? Unless I'm
%   not getting anything, either we change the notation of vectors to
%   parenthesis or else we have to explain here what this means
%   exactly.}  \marginch{$W_{\Nc}$ is a set. I define $W_{\Nc}$ later
%   to remove the confusion}
In this way, we obtain
\begin{subequations}
  \label{eq:constdis}
  \begin{alignat}{6}
    & \underline{P}_i \leq \tr\{Y_{i,r}W_i\} + P_{D_i} \leq \bar{P}_i,
    \; \forall i\in \mathcal{N}, \label{eq:Pbdd_r}
    \\
    & \underline{Q}_i \leq \tr\{\bar{Y}_{i,r}W_i\} + Q_{D_i} \leq
    \bar{Q}_i, \; \forall i\in \mathcal{N}, \label{eq:Qbdd_r}
    \\
    & \underline{V}_i^2 \leq \tr\{M_{i,r}W_i\} \leq \bar{V}_i^2, \;
    \forall i\in \mathcal{N}, \label{eq:VoltMag}
    \\
    & \tr\{M_{ik,r}W_k\} \leq \bar{V}_{ik}, \; \forall \{i,k\} \in
    \mathcal{E}.
    % & W_k \succeq 0, \;\; \text{rank}(W_k)=1 \;\; \forall k \in
    % \mathcal{N}, 
  \end{alignat}
\end{subequations}
Here, $Y_{i,r}$ is the principal submatrix of $Y_{i}$ obtained by
dropping the rows and columns associated with the buses in
$\N\setminus\N_i$. The matrices $\bar{Y}_{i,r}$, $M_{i,r}$, $M_{ik,r}$
are defined similarly.

Define $W_{\Nc}$ as a shorthand notation for the set of variables,
$W_i$ for $i\in\N$, i.e., $W_{\Nc}:= \{W_i,i\in\N\}$. We next consider
the following distributed convex optimization problem associated with
$W_{\Nc}$,
\begin{subequations}
  \begin{flalign}
    \textbf{(P2)} \hspace{10mm} \min_{W_{\mathcal{N}}} \sum_{i \in
      \N_G} & c_{i2}
    (\tr\{Y_{i,r}W_i\}+P_{D_i})^2 \label{eq:P2-cost} \\ \nonumber +
    &c_{i1} (\tr\{Y_{i,r}W_i\}+P_{D_i}), &&
  \end{flalign}
  subject to
  \begin{align}
    & \text{Eq.~\eqref{eq:constdis} holds}, \label{eq:P2-c1}
    \\
    &W_i\succeq 0, \quad \forall i\in\N, \label{eq:P2-c2}
    \\
    & W_i(\hat{i},\hat{i}) = W_k(\hat{i},\hat{i}), \; \forall \{i,k\}
      \in \E, \label{eq:linear-eq-P2-1}
    \\
    & W_i(\hat{i},\hat{k}) = W_k(\hat{i},\hat{k}), \; \forall \{i,k\}
      \in \E, \label{eq:linear-eq-P2-2}
  \end{align}
\end{subequations}
where $\hat{i}$ refers the row (or column) of the matrix associated
with bus~$i$ (note that $W_i$ and $W_k$ might have different
dimensions).
%
% \marginch{We need the hat because the dimension of $W_i$ and $W_k$
% are not the same ($N_i \neq N_k$) in general and the meaning of each
% entry is not that straight forward. For example, $W_i(1,1)$ and
% $W_k(1,1)$ may refer to the voltage of different buses. }
%
% We rewrite the equality constraints of $W_i$ and $W_k$,
% $\{i,k\}\in\mathcal{E}$, in the following matrix form for
% convenience
For chordal graphs, \textbf{(P2)} with the additional non-convex rank
constraints, $\rank(W_i)=1$, $\forall i\in \N$ is equivalent to
\textbf{(P1)}, see~\cite{SB-SHL-TT-BH:15}.  For general graphs,
\textbf{(P2)} with the rank constraints does not necessarily give an
optimal solution of \textbf{(P1)}, but simulations
indicate~\cite{CYC-WZ:16-2} that \textbf{(P2)} has a low-rank solution
whose value is close to the optimal value of~\textbf{(P1)}. In the
rest of the paper, we assume that a unique optimal solution of
\textbf{(P2)} exists, and we denote it as $(W^\star_{\N},p^\star)$.

Our objective in this paper is to design a distributed algorithm to
solve \textbf{(P2)}. We view each bus of the electrical network as a
computing agent that can communicate with any other bus which is
physically connected to. By distributed, we mean that each agent only
requires information from neighboring buses that share their local
variables to implement the algorithm. By solving the optimization
problem, we mean that each bus eventually finds its own optimal
allocation (not the optimal allocation for the whole electrical
network). When considered collectively, the local optimal allocation
with agreement on the shared variables yields the complete optimal
solution.
	
\section{The Scheduled-Asynchronous Algorithm}\label{sec:DisAlgo}

In this section, we first provide a design rationale for the
scheduled-asynchronous distributed algorithm and then introduce it
formally. We next proceed to characterize the algorithm convergence
properties. % regarding convergence and robustness to packet drops.
	
\subsection{Rationale for Algorithm Design}
	
For convenience of exposition, we start by rewriting the
optimization~\textbf{(P2)}. To this end, for each $i \in \Nc$, define
$\map{f_i}{\Hc^{N_i}}{\real}$ as the objective function, and let
$\W_i \subset \Hc^{N_i}$ be the constraint set defined
by~\eqref{eq:constdis} and the constraint $ W_i\succeq 0$. Note that
$\W_i$ is compact for every $i\in\N$, where the boundedness of $\W_i$
is the result of the bounded diagonal elements of $W_i$ and the
positive definiteness of $W_i$.  The objective function $f_i$ is given
by~\eqref{eq:P2-cost}, for $i\in\N_G$, and $f_i = 0$, for
$i\in\N\setminus\N_G$.  To represent the equality constraints
in~\eqref{eq:linear-eq-P2-1} and \eqref{eq:linear-eq-P2-2}, we
introduce the functions
$\map{G_{ik}}{\Hc^{N_i} \times \Hc^{N_k}}{\real^4}$,
\begin{align}
  \label{eq:EqulConstG}
  & G_{ik}(W_i,W_k) =  D_{ik}(W_i) +  D_{ki}(W_k),
\end{align}
where
\begin{align*}
  & D_{ki}(W_l) = \hspace{-1mm}\begin{bmatrix} \tr\{B_{1,ki}W_l\} \\
  \tr\{B_{2,ki}W_l\} \\
  \tr\{B_{3,ki}W_l\} \\
  \tr\{B_{4,ki}W_l\}
  \end{bmatrix}\hspace{-1mm},
  D_{ik}(W_l) = -\begin{bmatrix} \hspace{-1mm}\tr\{B_{2,ik}W_l\} \\
  \tr\{B_{1,ik}W_l\} \\
  \tr\{B_{3,ki}W_l\} \\
  \tr\{B_{4,ki}W_l\}
  \end{bmatrix}. \\
  & B_{1,ki}(l,m) =
  \begin{cases}
    1,  & \text{if } l=m=\hat{k},\\
    0, &  \text{otherwise},
  \end{cases}\displaybreak[0]
  \\ \nonumber
  & B_{2,ki}(l,m) =
  \begin{cases}
    1, &  \text{if $l=m=\hat{i}$,} \\ 
    0, & \text{otherwise},
  \end{cases} \displaybreak[0] \\ \nonumber 
  & B_{3,ki}(l,m) =
  \begin{cases}
    1, & \text{if $(l,m)=(\hat{k},\hat{i})$ or
      $(l,m)=(\hat{i},\hat{k})$,} \\
    0, & \text{otherwise},
  \end{cases} \displaybreak[0] \\ \nonumber
  & B_{4,ki}(l,m) =
  \begin{cases}
    -j, & \text{if $(l,m)=(\hat{k},\hat{i})$,} \\ 
    j,  &  \text{if $(l,m)=(\hat{i},\hat{k})$,} \\ 
    0,  & \text{otherwise}.
  \end{cases}
\end{align*}
%\margins{here is a good place to introduce the notation of the  $D_{ik}$. We could then express the $G_{ik}$ using those variables.}
Note that the linear equality
constraints~\eqref{eq:linear-eq-P2-1}-\eqref{eq:linear-eq-P2-2} can be
equivalently represented in compact form by $G_{ik}(W_i,W_k) = 0$, for
all $\{i,k\} \in \hat{\E}$, where~$\hat{\graph} = (\N,\hat{\E})$ is an
arbitrarily selected orientation of the original undirected
graph~$\graph = (\N,\E)$. With these elements in place, we rewrite
the optimization~\textbf{(P2)} in the following form
\begin{align}\label{eq:OPF_algo}
  &\min_{W_i\in\W_i, i \in \Nc } \sum_{i\in \N} f_i(W_i)
  \\
  \nonumber &\text{s.t. } G_{ik}(W_i,W_k) = 0, \; \forall
              \{i,k\}\in\hat{\E}.
\end{align}

To motivate our algorithm design, we start by considering the
optimization in \textbf{(P2)} for a two-bus network ($N=2$).  In this
case, from the formulation~\eqref{eq:OPF_algo}, the problem exactly
corresponds to the standard ADMM, see e.g.,~\cite{SB-NP-EC-BP-JE:11},
\begin{align*}
  &\min_{W_1\in\W_1, W_2\in\W_2} f_1(W_1) + f_2(W_2)\text{ \quad
    s.t. } G_{12}(W_1,W_2) = 0.
\end{align*}
The ADMM algorithm consists of the following steps
\begin{subequations}
  \label{eq:ADMM}
  \begin{alignat}{3}
    % \centering
    W_1^{t^+} &= \argmin_{W_1\in\W_1} f_1(W_1) \label{eq:ADMM-1}
    \\
    \nonumber &\hspace{-2mm}+ {p_{12}^t}^\top G_{12}(W_1,W_2^t) +
    \frac{\rho_{12}}{2} \|G_{12}(W_1,W_2^t)\|^2,
    \\
    W_2^{t^+} &= \argmin_{W_2\in\W_2} f_2(W_2) \label{eq:ADMM-2}
    \\
    \nonumber &\hspace{-2mm}+ {p_{12}^t}^\top G_{12}(W_1^{t^+},W_2) +
    \frac{\rho_{12}}{2} \|G_{12}(W_1^{t^+},W_2)\|^2,\label{eq:ADMM-3}
    \\
    p_{12}^{t^+} &= p_{12}^{t} + \rho_{12} G_{12}(W_1^{t^+},W_2^{t^+}),
  \end{alignat}
\end{subequations}
where the superscript $t$ is the time at which the update occurs,
$t^+$ is the time for the next round of the optimization,
$\rho_{12}>0$ is a given constant scalar, and
$p_{12}^t\in\mathbb{R}^4$ are the Lagrange multipliers
associated with the constraint $G_{12}(\cdot)=0$.  Note that there is a natural
order in performing the updates in~\eqref{eq:ADMM}, where node $1$
goes first, and then node $2$ uses the value obtained by $1$ to
perform its update.

For a network with an arbitrary number of buses, one can view the
optimization as a combination of multiple two-bus sub-problems. This
viewpoint inspires the following algorithm design.  Once bus $i$
receives the updated $W_k^t$ from all its neighboring nodes $k\in
\N_i$, it solves the following optimization
\begin{subequations}
	\label{eq:local_update}
\begin{alignat}{3}
  &W_i^{t^+} = \argmin_{W_i\in\W_i}f_i(W_i)\label{eq:local_update-1}
  \\
  &\;+\hspace{-3mm}\sum_{\{i,k\}\in \hat{\E}}\hspace{-3mm}\Big(
  {p_{ik}^t}^\top G_{ik}(W_i,W_k^t) +
  \frac{\rho_{ik}}{2}\|G_{ik}(W_i,W_k^t)\|^2 \Big)\label{eq:local_update-2} \\ 
  &\;+\hspace{-3mm}\sum_{\{k,i\}\in \hat{\E}}\hspace{-3mm}\Big(
  {p_{ik}^t}^\top G_{ki}(W_k^{t^+}\hspace{-1mm},W_i) +
  \frac{\rho_{ik}}{2}\|G_{ki}(W_k^{t^+}\hspace{-1mm},W_i)\|^2 \Big), \label{eq:local_update-3}
\end{alignat}
\end{subequations}
where $p^t_{ik}$ and $\rho_{ik}$ are non-directional, namely,
$p^t_{ik} = p^t_{ki}$ and $\rho_{ik} = \rho_{ki}$. Note that,
according to~\eqref{eq:local_update}, for every edge (i.e., for every
two-bus sub-problem), the tail node performs first the update,
followed by the head node. In other words, the orientation
$\hat{\graph} = (\N,\hat{\E})$ encodes the natural ordering of
updating by the terminal nodes present in the ADMM algorithm.
% Notice that by writing the updating rule in
% Eq.~\eqref{eq:local_update}, we have $\hat{\E}$ define the ordering
% of optimization of terminal nodes for every edge. The two nodes of
% each edge take turns in performing the optimization
% in~\eqref{eq:local_update}. Notice that,
Under the proposed design, the difference in the number of iterations
made between two connected nodes is at most one due to the alternating
execution.
% (the node that does~\eqref{eq:ADMM-1} goes first and hence may have
% % executed one more iteration than the other node at any given time)
% The tail node of each edge does the first step~\eqref{eq:ADMM-1} and
% the head node does the second step~\eqref{eq:ADMM-2}. 
For each pair of connected nodes $i$ and $k$ such that
$\{i,k\}\in\hat{\E}$, each of them updates the corresponding Lagrange
multiplier according to
\begin{align}\label{eq:lagran_update}
  p_{ik}^{t^+} = p_{ik}^t +
  \rho_{ik}G_{ik}(W_i^{t^+}, W_k^{t^+}) .
\end{align}
% \begin{subequations}\label{eq:lagran_update}
%   \begin{alignat}{2}
%     \text{Head node: } &p_{ik}^{t^+} = p_{ik}^t +
%     \rho_{ik}G_{ik}(W_k^{t^+}, W_i^{t^+}),\label{eq:lagran_update_1}
%     \\
%     \text{Tail node: } &p_{ik}^{t^+} = p_{ik}^t +
%     \rho_{ik}G_{ik}(W_i^{t^+}, W_k^{t^+}).\label{eq:lagran_update_2}
%   \end{alignat}
% \end{subequations}
Updating $p_{ik}^{t^+}$ locally at the terminal nodes can reduce the
communication burden and enhances robustness. We denote $p^t =
\{p^t_{ik}, \{i,k\}\in\E\}\in\real^{4|\E|}$ and
$\rho\in\real^{|\E|\times |\E|}$ be the diagonal matrix such that each
diagonal element corresponds to $\rho_{ik}$ of one unique link in
$\E$.  Algorithm~\ref{algo:pd} below presents formally the proposed
strategy.

\begin{algorithm}
  \caption{Scheduled-Asynchronous Algorithm}
  \begin{algorithmic}[1]
    % \Inputs {System matrices: $A,B,K$}
    \Initialize {$W_{\N}^0 \in\prod_{i\in\N}\W_i$, $p^0 = 0$,
      $\gamma^0_l = 2\epsilon>0,\; \forall l\in\N$ }
    \\
    \textbf{Requires:} acyclic orientation $\mathcal{\hat{G}}$ of
    $\graph$
    \\
    \textbf{For every bus $i$,}
    \While {(received $W_k^t,\gamma_k^t$
      from all $k\in\N_i$) \textbf{and} ($\exists l\in\N_i$
      s.t. $\gamma_l^t>\epsilon$) } 
    \State \textbf{Update} $p_{ik}^{t^+}$ by
    Eq.~\eqref{eq:lagran_update} for $\{i,k\}\in\hat{\E}$ \State
    \textbf{Update} $W^{t^+}_i$ by solving
    optimization~(\ref{eq:local_update}) 
    \State \textbf{Update} $p_{ik}^{t^+}$ by
    Eq.~\eqref{eq:lagran_update} for $\{k,i\}\in\hat{\E}$ 
    \State \textbf{Compute} $\gamma_i^{t^+}$ by
    Eq.~\eqref{eq:comp_error}
    \State \textbf{Send} $W_i^{t^+}$ and $\gamma_i^{t^+}$ to all
    $k\in\N_i$
    \EndWhile
    % \Outputs {optimality $\gamma_k$}
  \end{algorithmic}
  \label{algo:pd}
\end{algorithm}

In Algorithm~\ref{algo:pd}, each bus $i$ only does its optimization
after it has received new updates from all its neighbors $k\in\N_i$
since the last iteration.  The stopping criteria is given by the scale
of the violation of the equality constraints,
\begin{align}\label{eq:comp_error}
  \gamma_i^{t^+}\hspace{-2mm} = \hspace{-2mm}\sum_{\{i,k\}\in
    \hat{\E}}\hspace{-3mm}
  \|G_{ik}(W_i^{t^+}\hspace{-1mm},W_k^{t})\|^2 +
  \hspace{-3mm}\sum_{\{k,i\}\in \hat{\E}}\hspace{-1mm}
  \|G_{ki}(W_k^{t^+}\hspace{-1mm},W_i^{t^+})\|^2.
\end{align}
This criteria is justified by the observation that if
$\gamma_i^{t^+}=0$, $\forall i\in\N$, then $W_{\N}^{t^+}$ is the optimal
solution for both~\eqref{eq:local_update} and
\eqref{eq:OPF_algo}. Using the continuity of the cost function, $f_i$,
having $\gamma_i^{t^+}$ sufficiently small for all $i\in\N$ guarantees
that the solution of Algorithm~\ref{algo:pd} is reasonably close to
the optimum.  The implementation of Algorithm~\ref{algo:pd} does not
require synchronous updates between connected nodes, but involves a
subtle ordering. We therefore term this strategy as the
\emph{scheduled-asynchronous} algorithm. Note that we use a global
time index $t$ to time-stamp all the iterations in
Algorithm~\ref{algo:pd} only for convenience. Every agent tracks the
number of iterations locally, but in general does not know the global
time index.
	
\begin{remark}\longthmtitle{The orientation of the network
    graph must be acyclic}\label{rem:acyclic_graph}
  {\rm An important observation regarding the execution of
    Algorithm~\ref{algo:pd} is that the orientation
    $\mathcal{\hat{G}}$ given to the electrical network graph must be
    free of cycles. Otherwise, given the meaning encoded by the
    orientation of each edge, a deadlock would occur: every node at
    the cycle would be waiting for the update from a neighboring node
    in the cycle. There are various ways in which the network can
    determine an acyclic orientation in a distributed way. For
    instance, if every node has a unique identity $k\in\mathbb{N}$,
    then, for each edge in $\E$, one can designate the node with the
    smallest identity as the tail and the other node as the head.  The
    resulting graph $\mathcal{\hat{G}}$ is acyclic~\cite{RD:79}.  We
    revisit this point in Section~\ref{sec:Direct_graph}
    below.}\oprocend
\end{remark}

\begin{remark}\longthmtitle{The scheduled-asynchronous algorithm as a
    single-valued map}\label{rem:set-valued} {\rm The
    scheduled-asynchronous algorithm is, in general, a set-valued map
    due to the argmin operator in~\eqref{eq:local_update}. However, we
    argue here that it can  be seen as a single-valued map upon
    further examination of~\eqref{eq:local_update}.  Notice that,
    except for  the constraint $W_i\succeq 0$, all the other constraints and
    the objective function of~\eqref{eq:local_update} are only related to
    the entries associated with the star network centered at
    node~$i$. We illustrate the meaning of entries associated with a
    star 
    network in an  example with $5$ nodes. For $W_1$, its diagonal,
    first column, and first row elements are the entries associated
    with the star network centered at node one.
    \begin{align*}
      W_1&= \begin{bmatrix}
        W_{1}(1,1) & W_{1}(1,2) & \cdots & W_{1}(1,5) \\
        W_{1}(2,1) & W_{1}(2,2) & & W_{1}(2,5) \\
        \vdots & & \ddots & \vdots \\
        W_{1}(5,1) & W_{1}(5,2) & \cdots& W_{1}(5,5)
      \end{bmatrix} & \begin{tikzpicture}[baseline]
        [scale=.3,auto=right,every node/.style={circle,fill=blue!5}]
        \node (n1) at (0,0) {1}; \node (n2) at (.7,.7) {2}; \node (n3)
        at (-.7,.7) {3}; \node (n4) at (-.7,-.7) {4}; \node (n5) at
        (.7,-.7) {5}; \foreach \from/\to in {n1/n2,n1/n3,n1/n4,n1/n5}
        \draw (\from) -- (\to);
      \end{tikzpicture}
    \end{align*}
    We refer to the entries of $W_i$ that are not associated with the
    star network centered at node $i$ as ``irrelevant'', because those
    entries can take any value without affecting the optimal value
    of~\eqref{eq:local_update} as long as $W_i\succeq 0$ remains
    satisfied. Without loss of generality, Algorithm~\ref{algo:pd} can
    always assign zeros to the irrelevant entries. Such assignment
    makes $W_i$ a Hermitian matrix associated with a star network,
    and~\cite[Proposition 3]{CYC-WZ:16} ensures that $W_i\succeq 0$.
    The manipulation above makes the objective function
    of~\eqref{eq:local_update} strongly convex on the decision
    variables (irrelevant entries are considered as constants) due to
    the quadratic terms in \eqref{eq:local_update-2} and
    \eqref{eq:local_update-3}. This observation justifies the
    interpretation of the scheduled-asynchronous algorithm as a
    single-valued map $F:\W\times\real^{4|\hat{\E}|}\rightarrow
    \W\times\real^{4|\hat{\E}|}$ with
    $\W:=\W_1\times\W_2\times\cdots\times\W_N$.  }\oprocend
\end{remark}

\subsection{Convergence Analysis}
The analysis of the convergence properties of the
scheduled-asynchronous algorithm requires a careful consideration of
the asynchronous updates of the nodes. In what follows and for
convenience, we view the time index of the decision variables as an
iteration index. For each bus $i \in \N$, let $t_i(n)$ be the time
index at which $i$ has exactly performed $n$ number of the
minimizations described in~\eqref{eq:local_update}. By definition, for
each $n$, we have $W_i^t = W_i^{t_i(n)} $, for $t_i(n) \leq t <
t_i(n+1)$. With a slight abuse of notation, in the following we use
the shorthand notation $W_i^n$, $p_i^n$ and $r_i^n$ instead of the
corresponding $W_i^{t_i(n)}$, $p_i^{t_i(n)}$ and~$r_i^{t_i(n)}$, where
$r^t = \{r_{ik}^t \;,\; \{i,k\}\in\hat{\E}\}\in\real^{4|\hat{\E}|}$
and $r_{ik}^t = G_{ik}(W_i^t,W_k^t)$.
% Given any pair $(W_{\N},p)$, we will denote by $(W_{\N}^+,p^+)$ the
% pair obtained after all nodes have performed a first iteration
% starting from $(W_{\N},p)$; that is, if $(W_{\N}^0,p^0)=(W_{\N},p)$,
% then $(W_{\N}^+,p^+) = (W_{\N}^1,p^1)$.
%
With all the elements in place, we next establish the convergence
properties of Algorithm~\ref{algo:pd}.

\begin{theorem}\longthmtitle{Convergence of
    Algorithm~\ref{algo:pd}}\label{thm:converge}
  Assume the following conditions hold
  \begin{enumerate}
  \item the cost functions $f_i$, $i\in\N$, are convex,
  \item \textbf{(P2)} is feasible and Slater's condition holds,
  \item the optimal Lagrange multipliers are bounded, $p^\star<\infty$,
  \item the orientation $\hat\graph$ is acyclic.
  \end{enumerate}
  Then, the sequence $(W_{\N}^n,p^n)$ generated by
  Algorithm~\ref{algo:pd} converges to the optimal primal-dual pair
  $(W_{\N}^\star,p^\star)$ as $n\rightarrow\infty$.
\end{theorem}
\begin{proof}
  Our proof strategy consists of employing the LaSalle's Invariance
  Principle for discrete-time systems, [Theorem 1.19]
  in~\cite{FB-JC-SM:08cor}. To this goal, we will justify that all the
  assumptions of the LaSalle's theorem hold. First, according to
  Remark~\ref{rem:set-valued}, we can view $F$ as a single-valued
  mapping without loss of generality.  Furthermore, $F$ is continuous
  due to maximum theorem and condition 1).
  % In such case, we may not be able to show the convergence of
  % Algorithm~\ref{algo:pd} because it is necessary to resort to
  % set-valued LaSalle's invariance principle, which has weaker
  % convergence guarantees as the standard one. We address the issue
  % by the observation that
  % 
  % Hence, $F$ is a continuous single-valued mapping and satisfies the
  % continuity condition of [Theorem 1.19].
  % 
  % \marginch{The point is using LaSalle's theorem to prove the
  % convergence of the entries associated with the star
  % networks. Every other entry does not necessary converge to an
  % unique value. The good thing is that whatever value an irrelevant
  % entry takes (provided $W_i\succeq 0$), the final value
  % of~\eqref{eq:local_update} does not change and constraints are all
  % satisfied. }
  % 
  % \margins{note here that this map is singled-valued, as you pointed
  % out in the margin}
  % \margins{this mapping could be a set-valued mapping for general
  % $f_i$. I think we should remark after (11) that we assume this
  % mapping is a regular singled-valued mapping.}  \margins{OK, I
  % didn't find a good theorem that could give us the lipschitz
  % property we want. But we can use an theorem result ---I believe we
  % discussed this already some time ago--- since we have that the
  % solution is unique and $f$ is differentiable, the solutions are
  % given by $f'(x) =0$. Then, since the inverse exists, a $g(0) = x$,
  % then this $g$ has to be differentiable as well if $f'$ is
  % differentiable. Fromdifferentiability we have the locally
  % lipschitz property.}
  We next consider the candidate LaSalle function
  \begin{align}\label{eq:LaSalle}
    V(W_{\N},p) :=
    \hspace{-3mm}\sum_{\{i,k\}\in\hat{\E}}\hspace{-2mm}\bigg(
    \frac{\|p_{ik} -p^\star_{ik}\|^2}{\rho_{ik}} + \hspace{-1mm}
    \rho_{ik}\|D_{ki}(W_k - W_k^\star)\|^2\hspace{-1mm}\bigg).
  \end{align}
  For notational convenience, we write $V^n = V(W_{\N}^n,p^n)$. Recall
  that we assume $p^\star$ is bounded in condition 3) and
  $W_{\N}^\star$ is also bounded because $\W$ is compact. It follows
  that $V^0 <\infty$ for any bounded initial $(W^0_{\N},p^0)$.

  \emph{Monotonicity of LaSalle function}. We next show that $V$ is
  monotonically non-increasing along the solutions
  of~\eqref{eq:local_update}-\eqref{eq:lagran_update}, specifically,
  \begin{align}
    \label{eq:prove_3}
    V^{n+1} \hspace{-2mm}- V^n \leq- \hspace{-2mm}
    \sum_{\{i,k\}\in\hat{\E}}\hspace{-2mm}\rho_{ik}\|r_{ik}^{n+1}\hspace{-2mm}
    -D_{ki}(W_k^{n+1}\hspace{-1mm} - W_k^n)\|^2.
  \end{align}
  To show this inequality, we first sum the inequalities in
  Lemma~\ref{lem:pre_proof} to obtain
  \begin{align*}
    & \hspace{-2.0mm} \sum_{\{i,k\}\in\hat{\E}}
    \hspace{-2.5mm}\rho_{ik} D_{ki}^\top(W_k^{n+1}\hspace{-3mm} -
    W_k^n)D_{ik}(W_i^{n+1} \hspace{-3mm}- W_i^\star) \geq
    (p^{n+1}\hspace{-3mm} -p^{\star})^\top \hspace{-1mm}r^{n+1}.
  \end{align*}
  Using $r^{n+1} =\rho^{-1}(p^{n+1} -p^n)$ and $r^\star = 0$, we
  rewrite the inequality above as
  \begin{align*}
    &2\hspace{-2mm}\sum_{\{i,k\}\in\hat{\E}}\hspace{-2mm}\rho_{ik}
    D_{ki}^\top(W_k^n - W_k^{n+1}) \Big( D_{ki}(W_k^{n+1}\hspace{-2mm}
    - W_k^\star)- r_{ik}^{n+1} \Big)
    \\
    & \hspace{5mm} \geq 2(p^{n+1} -p^{\star})^\top\rho^{-1}
    (p^{n+1} -p^n ).
  \end{align*}
  Using the fact that
  \begin{align*}
    &2(p^{n+1} -p^{\star})^\top \rho^{-1}(p^{n+1} -p^n
    )=\|\sqrt{\rho}^{-1}(p^{n+1} -p^{\star})\|^2
    \\
    &\quad + \|\sqrt{\rho}^{-1}(p^{n+1} -p^n)\|^2 -
    \|\sqrt{\rho}^{-1}(p^{n} -p^{\star})\|^2,
    \\
    &2D_{ki}^\top(W_k^n - W_k^{n+1})D_{ki}(W_k^{n+1}\hspace{-2mm} -
    W_k^\star)= \|D_{ki}(W_k^n - W_k^\star)\|^2
    \\
    & \quad - \|D_{ki}(W_k^{n+1}\hspace{-2mm} - W_k^{n})\|^2 -
    \|D_{ki}(W_k^{n+1}\hspace{-2mm} - W_k^\star)\|^2,
  \end{align*}
  where $\sqrt{\rho}$ denotes the element-wise square root of the
  diagonal matrix $\rho$, then
  \begin{align*}
    &-\hspace{-2mm}\sum_{\{i,k\}\in\hat{\E}}\hspace{-1.5mm}\rho_{ik}\bigg(
    \|D_{ki}(W_k^{n+1} \hspace{-3mm}- W_k^\star)\|^2 \hspace{-1mm}-
    \|D_{ki}(W_k^n - W_k^\star)\|^2
    \\
    &+\|D_{ki}(W_k^{n+1}\hspace{-3mm} - W_k^n)\|^2\hspace{-1mm} +
    2D_{ki}^\top(W_k^n - W_k^{n+1}) r_{ik}^{n+1}\hspace{-1mm}\\ +
    &\|r^{n+1}_{ik} \|^2 \bigg) \geq \|\sqrt{\rho}^{-1}(p^{n+1}
    -p^{\star})\|^2 - \|\sqrt{\rho}^{-1}(p^n -p^{\star})\|^2.
  \end{align*}
  Using the definition~\eqref{eq:LaSalle} of $V$, we can identify the
  terms $V^n$ and $V^{n+1}$ in the inequality above to obtain
  \begin{align}
    &-\hspace{-2mm}\sum_{\{i,k\}\in\hat{\E}}\rho_{ik}\bigg(\|D_{ki}(W_k^{n+1}
    - W_k^n)\|^2+ 2{r_{ik}^{n+1}}^\top
    \\
    \nonumber & \cdot D_{ki}(W_k^n - W_k^{n+1}) +\|r^{n+1}_{ik}
    \|^2\bigg) \geq V^{n+1}\hspace{-1mm}-V^n
  \end{align}
  Rearranging the left-hand side leads to~\eqref{eq:prove_3}.

  \emph{Bounded Trajectories}.  We next justify that the trajectories
  of Algorithm~\ref{algo:pd} are bounded.  According
  to~\eqref{eq:local_update_iter}, the primal variable $W^n_{\N}$
  always evolves in a compact set $\W$ and therefore is bounded.  To
  show that the evolution of $p^n$ is also bounded, we can reason by
  contradiction. If it were not, then the sequence $V(W^n,p^n)$ would
  go to infinity, and this would contradict the fact that the sublevel
  sets of $V$ are invariant (which is as a consequence
  of~\eqref{eq:prove_3}).

  \emph{Application of LaSalle Invariance Principle.}  Given our
  discussion above, all assumptions of the LaSalle Invariance
  Principle~\cite{FB-JC-SM:08cor} hold and we conclude that as $n\rightarrow
  \infty$, $(W^n_{\N},p^n)$ converges to the largest invariant set
  $\I_I$ contained in $\I_0$, where
  \begin{align}\label{eq:I0}
    \I_0:=\{(W_{\N},p) \,| \, V(F(W_{\N},p))-V(W_{\N},p)=0 \}.
  \end{align} 
  Our final step to establish the result is to show that $\I_I = \{(W_{\N},p)|V(W_{\N},p)=0\}$. To this end, let $(W^0_{\N},p^0)$ be an
  arbitrary point in $\I_I$. Consider the algorithm trajectory starting
  from $(W^0_{\N},p^0)$, which by definition of the notion of
  invariance must remain in $\I_0$.  The next equalities must hold
  because of the definition of $\I_0$,
  \begin{subequations}\label{eq:invar_sup}
    \begin{alignat}{3}
      &r^{n+1}_{ik} = D_{ki}(W_k^{n+1}\hspace{-1mm} -
      W_k^n), \label{eq:invar_sup1}
      \\
      \Longleftrightarrow \; & D_{ik}(W_i^{n+1}) = -
      D_{ki}(W_k^{n}), \label{eq:invar_sup2}
      \\
      \Longleftrightarrow \; & r^{n+1}_{ik} =
      D_{ik}(W_i^{n+1}\hspace{-1mm} - W_i^{n+2}),\label{eq:invar_sup3}
    \end{alignat}
  \end{subequations}
  for all $\{i,k\}\in\hat{\E}$ and $n\geq 0$ because the right-hand
  side of~\eqref{eq:prove_3} should be zero.  Let $\sink_0$ be
  the set of sink nodes in~$\hat{\graph}$. The primal variable update
  for $i\in\sink_0$ is given as
  \begin{subequations}
    \label{eq:iter_sink}
    \begin{alignat}{2}
      &W_i^{n+1} = \argmin_{W_i\in\W_i}f_i(W_i)\label{eq:iter_sink1} 
      \\
      \nonumber &\;\;+\hspace{-2mm}\sum_{\{k,i\}\in
        \hat{\E}}\hspace{-2mm}\Big( {p_{ik}^{n}}^\top
      G_{ki}(W_k^{n+1},W_i) +
      \frac{\rho_{ik}}{2}\|G_{ki}(W_k^{n+1},W_i)\|^2 \Big) 
      \\
      &\hspace{10mm}=
      \argmin_{W_i\in\W_i}f_i(W_i) \label{eq:iter_sink2}
      \\
      \nonumber &\hspace{10mm}+\hspace{-2mm}\sum_{\{k,i\}\in
        \hat{\E}}\hspace{-2mm}\Big(\frac{\rho_{ik}}{2}\|G_{ki}(W_k^{n+1},W_i)
      + \frac{p_{ik}^{n}}{\rho_{ik}}\|^2 \Big).
    \end{alignat}
  \end{subequations}
  Eq.~\eqref{eq:iter_sink2} follows by completing the squares inside
  the sum over $\hat{\E}$ in~\eqref{eq:iter_sink1} (we also drop the
  term ${{p_{ik}^n}^2}/{2\rho_{ik}}$ because this does not affect the
  argmin operation).
  % \margins{I see how the square of the term inside the sum in a.36b
  % gives you the two terms inside the sum of a.36a. And I suppose
  % that the term $\frac{{p_{ik}^n}^2}{2 \rho_{ik}}$ in a.36b does not
  % affect the argmin, right? I would then say that the equation a.36b
  % is derived by completing the squares inside the sum over
  % $\hat{\E}$ and noting that the argmin operation does not change by
  % the addition of the term $\frac{{p_{ik}^n}^2}{2 \rho_{ik}}$}
  Now we analyze how the last quadratic term in~\eqref{eq:iter_sink2}
  evolves as $n$ increases
  \begin{align}\label{eq:evol_n}
    & \frac{\rho_{ik}}{2}\|G_{ki}(W_k^{n+1},W_i) +
    \frac{p_{ik}^{n}}{\rho_{ik}}\|^2 \\
    & = \frac{\rho_{ik}}{2}\|D_{ki}(W_k^{n+1}) + D_{ik}(W_i) +
    \frac{p_{ik}^{n}}{\rho_{ik}}\|^2 \nonumber \\
    & = \frac{\rho_{ik}}{2}\|(D_{ki}(W_k^{n}) - r_{ki}^{n}) +
    D_{ik}(W_i) + \frac{p_{ik}^{n-1} + \rho_{ik}r_{ki}^{n}
    }{\rho_{ik}}\|^2 \nonumber
    \\
    & = \frac{\rho_{ik}}{2}\|G_{ki}(W_k^{n},W_i) +
    \frac{p_{ik}^{n-1}}{\rho_{ik}}\|^2. \nonumber
  \end{align} 
  We use~\eqref{eq:invar_sup3} and~\eqref{eq:lagran_update} in the
  second equality in~\eqref{eq:evol_n}, with node $k$ being the tail
  node as $\{k,i\}\in\hat{\E}$ % instead of
  % $\{i,k\}\in\hat{\E}$ in Eq.~\eqref{eq:invar_sup3} and
  % \eqref{eq:lagran_update}
  \footnote{Recall that $p_{ik}^n$ and $\rho_{ik}$ are
    non-directional, i.e., $p^n_{ik} = p^n_{ki}$ and $\rho^n_{ik} =
    \rho^n_{ki}$. However, $r^n_{ik} = G_{ik}(W_i^n,W_k^n) =
    D_{ik}(W_i^n) + D_{ki}(W_k^n)$ are directional.}. Note
  that~\eqref{eq:evol_n} holds only for $n\geq 1$
  because~\eqref{eq:invar_sup3} holds only for $n\geq 0$.  Due
  to~\eqref{eq:evol_n}, optimization~\eqref{eq:iter_sink} for node
  $i\in\sink_0$ does not change with respect to the iteration number
  as long as $n\geq
  1$. %\margins{this notation for sink nodes does not have subzero as
      %before. Explain why this applies to sink nodes only (say more
      %explicitly why for a non-sink node you can not use the same
      %argument right now)} \marginch{i add explanation why the
      %argument does not apply to non-sink node shortly after.}
  Hence,
  \begin{align}\label{eq:pre_invar}
    D_{ik}(W_i^{n+1} - W_i^{n}) = 0 \text{ and } r_{ki}^{n+1} = 0,
  \end{align}
  as a result of Eq.~\eqref{eq:invar_sup3}, for all $i\in\sink_0$ and
  $\forall n\geq 1$. Notice that we have \eqref{eq:pre_invar} only for
  $i\in\sink_0$ because if $i\not\in\sink_0$, \eqref{eq:iter_sink}
  includes additional terms for $\{i,k\}\in\hat{\E}$, for which
  results similar to~\eqref{eq:evol_n} are not
  available. 
  % \margins{shouldn't we be able to say from here that $W_i^{n+1} =
  % W_i^n$? After all, the solution is a single point, the quadratic
  % terms do no change and neither does the $f_i$
  % change. }\marginch{Not there yet, at this point we have
  % \eqref{eq:pre_invar} only for links connected to sink nodes, which
  % is a subset of $\hat{\E}$. }
	
  Next consider the subgraph of $\hat{\graph}$ induced by the set of
  vertices $\N\setminus\sink_0$, denoted as
  $\hat{\graph}[\N\setminus\sink_0]$. The graph
  $\hat{\graph}[\N\setminus\sink_0]$ may be composed of several
  disconnected subgraphs in general. Every subgraph of
  $\hat{\graph}[\N\setminus\sink_0]$ has at least one sink node
  because $\hat{\graph}[\N\setminus\sink_0]$ remains acyclic. Let
  $\sink_1$ be the set of sink nodes of
  $\hat{\graph}[\N\setminus\sink_0]$. The optimization on the primal
  variables of $i\in\sink_1$ can be written as
  \begin{align}\label{eq:iter_sink_source}
    &W_i^{n+1} = \argmin_{W_i\in\W_i}f_i(W_i)
    \\
    \nonumber &\;\;+\hspace{-2mm}\sum_{\{i,k\}\in \hat{\E}, k\in\sink_0}\hspace{-2mm}
    \frac{\rho_{ik}}{2}\|G_{ik}(W_i,W_k^n) + \frac{p_{ik}^n}{\rho_{ik}}\|^2 \\ \nonumber
    &\;\;+\hspace{-2mm}\sum_{\{k,i\}\in \hat{\E}}\hspace{-2mm}
    \frac{\rho_{ik}}{2}\|G_{ki}(W_k^{n+1},W_i) + \frac{p_{ik}^n}{\rho_{ik}}\|^2.
  \end{align}
  We can show that the quadratic-cost functional terms associated with
  $\{k,i\}\in\hat{\E}$ do not change for $n\geq 1$ because of
  Eq.~\eqref{eq:evol_n}, which holds for any $\{k,i\}\in
  \hat{\E}$. The quadratic-cost functional terms associated with links
  $\{i,k\}\in\hat{\E}$, for $k\in\sink_0$, do not change either for
  $n\geq 1$ because Eq.~\eqref{eq:pre_invar} holds for every
  $k\in\sink_0$ for $n\geq 1$.  Therefore, we have
  Eq.~\eqref{eq:pre_invar} satisfied for all $i\in\sink_1$ for every
  $n\geq 2$.  With similar arguments for the subgraph
  $\hat{\graph}[\N\setminus(\sink_0\cup\sink_1)]$, we can conclude
  that Eq.~\eqref{eq:pre_invar} holds for nodes in $\sink_2$ within
  finite iterations, where $\sink_2$ is the set of sink nodes of
  $\hat{\graph}[\N\setminus (\sink_0\cup \sink_1)]$. Repeating the
  induction, it follows that after a finite number of iterations,
  Eq.~\eqref{eq:pre_invar} holds for all $i\in\N$. The
  number of iterations required for Eq.~\eqref{eq:pre_invar} to hold
  is the diameter of the directed graph $\hat{\graph}(\N,\hat{\E})$.
  
  % \margins{I believe that you have proven that $W^{n+1} = W^n$ and
  % that $p^{n+1} = p^n$ for sufficiently large and finite
  % $n$. (though you don't state this, but you state that A.38 holds
  % only.) I must be missing something. }\marginch{yes, we can almost
  % conclude $W^{n+1} = W^n$ here. It's only a subtle difference
  % between A.40 and $W^{n+1} = W^n$, which is explained in the proof
  % of the theorem}
  % \margins{IN any case, how do you get from here to say that $V =0$?
  % I think these steps are in the theorem proof, but not here. Then,
  % I would rephrase the lemma to just state that eq A.38
  % holds. }\marginch{i have rephrased the lemma}
  In the reasoning above, we have shown that there exists 
  $n_0<\infty$ such that 
   \begin{align}
     \label{eq:converge_diff2}
     \|r^{n_0+1}_{ik}\| = \|D_{ki}(W_k^{n_0+1}\hspace{-1mm} -
     W_k^{n_0})\|^2 = 0,
   \end{align}
   holds for all $\{i,k\}\in\hat{\E}$. Notice that $r^{n_0+1} = 0$
   indicates the solution of~\eqref{eq:local_update_iter} for
   iteration $n_0+1$ is $W^\star_{\N}$. Thus, $W^{n_0+1}_{\N}=
   W^\star_{\N}$, and $p^{n_0+1} = p^{n_0}= p^\star$ because of strong
   duality. Therefore, we have $V^{n_0+1} =
   V(W^{n_0+1}_{\N},p^{n_0+1}) = V(W^{\star}_{\N},p^{\star}) = 0$. By
   definition of $\I_0$, $V$ remains constant throughout the trajectory starting from $(\W_{\N}^0,p^0)$. Therefore, we can conclude that the largest invariant set is 
   \begin{align}\label{eq:invar_pre}
	\setdef{(W_{\N},p)}{V(W_{\N},p)= 0}.
   \end{align}
   By LaSalle theorem, we next conclude that as $n\rightarrow\infty$
   \begin{align}\label{eq:converge}
   D_{ki}(W^n_k) \rightarrow D_{ki}(W_k^\star),\quad p^n_{ik} \rightarrow p_{ik}^\star,
   \end{align}
   for all $\{i,k\}\in\hat{\E}$. Combining~\eqref{eq:converge} and the optimization steps in~\eqref{eq:local_update}, we also have that for all $\{i,k\}\in\hat{\E}$
   \begin{align}\label{eq:converge2}
	D_{ik}(W^n_i)\rightarrow D_{ik}(W_i^\star),
	\end{align}   
	as $n\rightarrow\infty$. 
   The equalities~\eqref{eq:converge} and \eqref{eq:converge2} hold if and only if $(W^n_{\N},p^n) \rightarrow
   (W_{\N}^\star,p^\star)$, completing the proof.
\end{proof}

% In the proof of Theorem~\ref{thm:converge}, we have parts of the
% entries of $W^n_{\N}$ being zero to ensure the continuity of the
% mapping $F$. It is worth to point out that when $(W^n_{\N},p^n)$
% converges and if the optimum of \textbf{(P2)} has the rank being
% one, we can always retrieve the rank one optimum from the limit
% point of $W^n_{\N}$ by replacing the zero entries of $W^n_{\N}$ to
% values such that $\rank(W^\star_{i})=1$ for all $i\in\N$.

\begin{remark}\longthmtitle{Convergence rate of
    Algorithm~\ref{algo:pd}}\label{rem:convergence_rate}
  {\rm Characterizing the convergence rate of the
    scheduled-asynchronous algorithm is challenging. The main reason
    for this is the fact that, given the particular characteristics of
    the design of Algorithm~\ref{algo:pd}, the proof of
    Theorem~\ref{thm:converge} identifies a LaSalle function, and not
    a Lyapunov one. This is in contrast with the convergence proof of
    ADMM, cf.,~\cite{BH-XY:15}, where the availability of a Lyapunov
    function makes the convergence analysis and the rate
    characterization easier.  If the network topology is bipartite,
    Algorithm~\ref{algo:pd} reduces to ADMM and hence shares the
    $O(1/n)$ convergence rate established in~\cite{BH-XY:15}. Based on
    this and our simulation results, we conjecture that
    Algorithm~\ref{algo:pd} also has $O(1/n)$ convergence rate with a
    smaller constant.}\oprocend
\end{remark}
	
\begin{remark}\longthmtitle{Heterogeneous selection of
    parameter~$\rho$}\label{rem:para_selection} {\rm The convergence
    rate of Algorithm~\ref{algo:pd} depends on the parameter~$\rho$
    which weighs how the errors in the constraint satisfaction affect
    the evolution of the dual variables. Our simulation studies show
    that selecting $\rho_{ik}$ heterogeneously can improve the
    convergence rate dramatically. In particular, we observe that
    choosing $\rho_{ik}$ proportionally to the norm of the complex
    line admittance of edge $\{i,k\}$, $\|y_{ik}\|$ improves the
    convergence rate compared to selecting a uniform $\rho$, i.e.,
    $\rho_{ik} = \rho_0, \forall\{i,k\}\in\E$. The intuition behind
    this fact is that the network is more sensitive to the
    perturbation on the edges with larger admittances. To justify this
    point, consider an edge where $G_{ik}$ in~\eqref{eq:EqulConstG} is
    not zero, \begin{align}\label{eq:perturb_con}
      G_{ik}(W_i^t,W_k^t)=\delta_{ik}^t,
      \;\forall\{i,k\}\in\hat{\E}, \end{align} with $0 \neq
    \delta_{ik}^t\in\real^4$.  % Consider at time $t$ before the convergence, the terminal nodes $i$
    % and $k$ can view~\eqref{eq:perturb_con} as a $\delta_{ik}^t$
    % perturbation on the constraint, where $\delta_{ik}^t\in\real^4$.
    Due to~\eqref{eq:perturb_con}, $W_{\N}^t$ is not a feasible
    solution of \textbf{(P2)}.  Consider then the virtual network in
    Fig.~\ref{fig:virtual_network}, where nodes $i$ and $k$ are not
    physically connected and only the communication between them
    remains.  %
    \begin{figure} \centering%
      \includegraphics[width=0.35\textwidth]{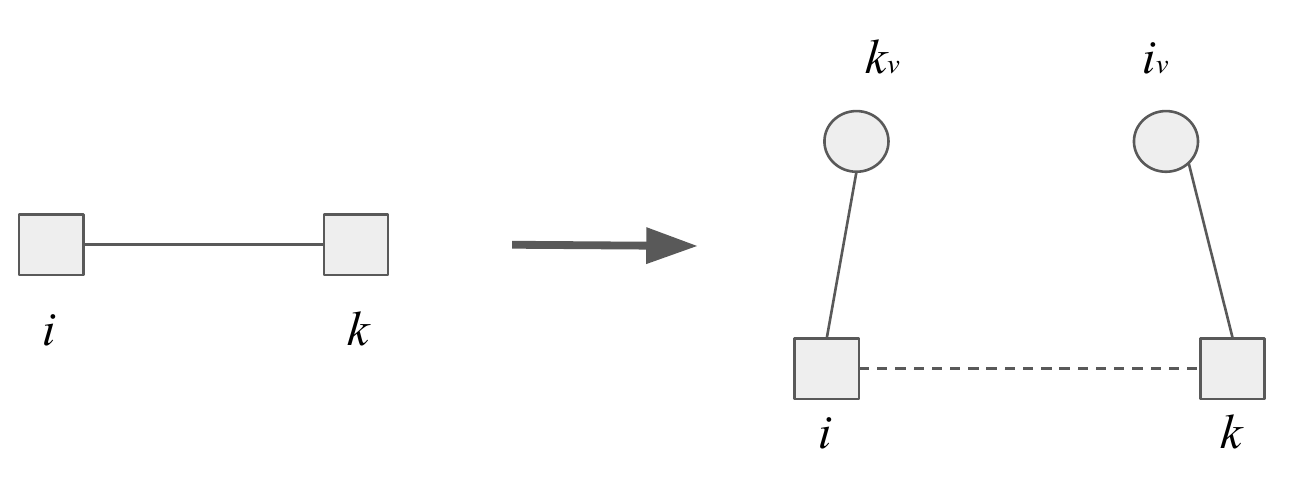} %
      \caption{Illustration of the virtual subnetwork with
        perturbation on the constraint $G_{ik}(W_i,W_k)=0$.  The solid
        lines represent the physically connected edge with
        communication. The dash line corresponds to a communication
        link.}\label{fig:virtual_network}
    \end{figure}
    Node $i$ is
    connected to a virtual node $k_v$ with admittance $y_{ik}$, and
    the same applies to $k$ and $i_v$. The communication between $i$
    and $k$ have $W_i$ and $W_k$ satisfy
    $G_{ik}(W_i^t,W_k^t)=\delta_{ik}^t$.
% Here, for the simplicity of analysis, we assume the second largest
% eigenvalue of every $W_i^t$ is small compared to its largest one so
% that the approximation $\rank(W_i^t)=\rank(W_k^t)=1$ is valid even
% $\delta^t_{ik}\neq 0$.
% Recall that $W_i$ encodes one copy of voltages for all nodes
% $l\in\N_i$.
    Let $V_{i\hat{i}}$ and $V_{i\hat{k}}$, respectively, denote the copy
    of the voltages of bus $i$ and $k$ contained in $W_i$. Assign to nodes
    $k_v$ and $i_v$ the voltages $V_{i\hat{k}}$ and $V_{k\hat{i}}$,
    respectively. In this way, the virtual network has the power flow from
    $k_v$ to $i$ given by
    \begin{align}\label{eq:kv_i}
      V_{i\hat{k}}\Big(y_{ik}(V_{i\hat{k}} - V_{i\hat{i}})\Big)^*.
    \end{align}
    On the other hand, the power flow from $k$ to $i_v$ is
    \begin{align}\label{eq:k_iv}
      V_{k\hat{k}}\Big(y_{ik}(V_{k\hat{k}} - V_{k\hat{i}})\Big)^*.
    \end{align}
    The copies of voltages of nodes $i$ and $k$ are related by
    \begin{align}\label{eq:v_volts}
      V_{i\hat{k}}^t = V_{k\hat{k}}^t + \hat{\delta}_{ik,k}^t, \quad
      V_{i\hat{i}}^t = V_{k\hat{i}}^t + \hat{\delta}_{ik,i}^t, 
    \end{align}
    where $\hat{\delta}^t_{ik,k} = \sqrt{|\delta_{ik}^t(1)|}
    \exp(j\angle{(\delta^t_{ik}(3) + j\delta^t_{ik}(4))})$ and
    $\hat{\delta}^t_{ik,i} = \sqrt{|\delta^t_{ik}(2)|}$ (assume
    $\angle{V_{i}}=0$ without loss of generality). If
    $\delta_{ik}^t=0$, then the values of~\eqref{eq:kv_i}
    and~\eqref{eq:k_iv} are the same because $V_{i\hat{k}} =
    V_{k\hat{k}}$ and $V_{i\hat{i}}=V_{k\hat{i}}$. Therefore, the
    virtual network is equivalent to the original one if
    $\delta_{ik}^t=0$. On the other hand, if the first two elements of
    $\delta_{ik}^t$ are non-zero, then $V_{i\hat{k}}^t\neq
    V_{k\hat{k}}^t$ and $V_{i\hat{k}}^t\neq V_{k\hat{i}}^t$ according
    to~\eqref{eq:v_volts}. In such case, the virtual network is not
    equivalent to the original one. Furthermore, we observe that for a
    given non-zero $\delta_{ik}^t$, the discrepancy between
    Eq.~\eqref{eq:kv_i} and \eqref{eq:k_iv} is proportional to
    $y_{ik}$.  This discrepancy has the interpretation of a
    perturbation on the apparent power flow for edge $\{i,k\}$.
  }\oprocend
\end{remark}

\section{Directed Graph Design}\label{sec:Direct_graph}

Here, we first describe how the average time between two consecutive
iterations under the scheduled-asynchronous algorithm for a node
depends on the diameter of the orientation of the network
graph. Motivated by this observation, we set out to find an acyclic
orientation with minimal diameter in a distributed fashion. In
general, finding such an orientation is equivalent to finding the
chromatic number, which is NP-hard. However, exploiting the planar
property of many power networks, we develop a distributed algorithm to
determine an orientation with small diameter.
	
\subsection{Relationship Between Graph Coloring and Diameter}  
The updating sequence of the scheduled-asynchronous algorithm depends
on the orientation $\hat{\graph} = (\N,\hat{\E})$.  To see this,
consider a path in $\hat{\graph}$, with $i$ being the initial node and
$k$ the end node. The update of $k$ requires that all other nodes in
the path finish at least an update before it, which are processed in
sequence starting from node $i$.  The ``waiting time'' of node $k$
associated to this path corresponds therefore to its length. Since the
diameter of the orientation bounds the length of all the paths, this
justifies finding an acyclic orientation which the smallest
diameter. We formalize this problem next.  Let $\Omega (\graph)$ be
the collection of all acyclic orientations of $\graph$.  We denote by
$\graph_{\omega}$ the directed graph corresponding to $\omega \in
\Omega (\graph)$.  The problem we aim to solve is % of finding an acyclic orientation that
% minimizes the diameter is
\begin{align}
  \label{eq:min_diameter}
  \omega^\star =
  \text{arg}\min_{\omega\in\Omega  (\graph)} \big( \max_{h\in\mathcal{P}_\omega}|h|
  \big),
\end{align}
where $\mathcal{P}_\omega$ is the set of paths in $\graph_\omega$, and
$|h|$ is the length of the path~$h$. The
optimization~\eqref{eq:min_diameter} is directly related to the
classical problem of finding the chromatic number $\X(\graph)$ of an
undirected graph $\graph$~\cite{RF-VB-NM-CS:08}. In fact, one has
\begin{align}
  \label{eq:diameter_chrom}
  \X(\graph) = 1 +  \min_{\omega\in\Omega  (\graph)} \big(
  \max_{h\in\mathcal{P}_\omega}|h| \big).
\end{align}
% Based on Eq.~\eqref{eq:diameter_chrom}, the problem is reduced to
% finding the chromatic number $\X(\graph)$.  However, 
In general, computing $\X(\graph)$ is NP-hard, see
e.g.~\cite{AS:89}. There are only approximate algorithms to find a
solution, see for example~\cite{DC-BG:73}.  However, for planar
graphs, the chromatic number is upper bounded by
four~\cite{NR-DS-PS-RT:97} and, furthermore, there exists a
quadratic-time algorithm to find a graph-coloring with four colors,
cf.~\cite{NR-DS-PS-RT-96}.  Fortunately, the following assumption
holds for most electrical networks.
\begin{assumption}\longthmtitle{Planar network topology~\cite{KCS:15}}
  \label{assum:topo}
  Electrical networks have simple planar network topology.
\end{assumption}
If a centralized entity has information on the network topology, then
the algorithm in~\cite{NR-DS-PS-RT-96} can be run to assign a number
(or color) to every agent. % If the network topology recognized by the
% centralized entity is not far-off from the actual one, then the
% acyclic graph defined in this way is near optimal.
This procedure might be problematic in large-scale scenarios,
specially in the presence of plug-and-play devices that easily change
the network topology. This motivates our design of a distributed
algorithm.

\subsection{Distributed Orientation Computation with Small Diameter} 

The following distributed algorithm from~\cite{SG-MHK:93} finds an
orientation of an arbitrary planar graph with diameter bounded by five
(or equivalently $\X(\graph)\leq 6$):
\begin{enumerate}
\item find an acyclic orientation in a distributed way such that every
  node has out-degree at most five.
\item every node chooses a color that is different from all the
  out-neighboring nodes defined in step 1).
\end{enumerate}
Note that any ordering of the colors induces an orientation of the
graph, which is furthermore acyclic,
cf. Remark~\ref{rem:acyclic_graph}.
% Note that the out-degree is defined for directed graphs, which is
% distinguish from the concept of degree for undirected graphs. 
The algorithm above is simple to implement but may be conservative for
electrical networks because, in general, the degree of most nodes is
far less than six. In fact, empirical
studies~\cite{PH-SB-ECS-CB:10,RA-IA-GLN:04} have shown that electrical
networks have a degree distribution that follows the exponential
distribution. Since the degree takes integer values, it is more
appropriate to characterize the degree distribution with the geometric
distribution -- the discrete analogy of the exponential distribution.

\begin{assumption}\longthmtitle{Geometric degree
    distribution}\label{assum:exp_deg_dist}
  In electrical networks, the number of nodes with degree at least
  $d_0\in\Pint$ satisfies
  \begin{align}
    \label{eq:deg_dist}
    \text{Prob}(d = d_0)= \lambda(1-\lambda)^{d_0}, 
  \end{align}
  with parameter $0\leq\lambda\leq 1$.
\end{assumption}
% %
% \marginJC{So different electrical networks have different $\lambda$?
%   Can you specify the right lambdas for ieee 14, 30, 57?}
% \marginch{yes different networks have different $\lambda$. I tried to
%   curve fitting IEEE 14 30 57 examples with geometric
%   distribution. The best values of lambda are 0.23, 0.27, and 0.24,
%   but the errors of the curve fitting are not small. The errors mainly
%   come from $d=1$ and $d=2$. The degree follows geometric distribution
%   very well for $d>=2$, but the number of nodes with $d=1$ is far less
%   than the number for $d=2$, making geometric (or exponential
%   distribution) not very competent. Discrete Weibull distribution fits
%   the data better.}
% %
Most nodes have small degree due to~\eqref{eq:deg_dist}. According
to~\cite{PH-SB-ECS-CB:10,WD-WL-XC-QAW:11}, the average degree of
empirical electrical networks is between $2$ and $3$. Many electrical
networks, as a result, have the chromatic number far less than
six. For example, IEEE 14, 30 and 57 bus test cases have $\X(\graph)
= 3$.  Therefore, it is conceivable that one can find an orientation
with diameter less than five. Motivated by this observation, we modify
the algorithm in~\cite{SG-MHK:93} to exploit the geometric degree
distribution property.

We propose Algorithm~\ref{algo:degree} to find a preliminary
orientation.
% Algorithm~\ref{algo:degree} operates under
% Assumptions~\ref{assum:topo} and~\ref{assum:exp_deg_dist} so that it
% may find a initiatory orientation with the out-degree less than one
% of the algorithm in~\cite{SG-MHK:93}.
The idea is to first try to find an orientation with the out-degree of
all nodes no bigger than 2.  If this is not possible, then the
strategy searches instead for an orientation with one more
out-degree. If necessary, this process is repeated until
Algorithm~\ref{algo:degree} eventually finds an orientation where the
out-degree of all nodes is at most five.

\begin{algorithm}[h!]
  \caption{}
  \begin{algorithmic}[1]
    % \Inputs {System matrices: $A,B,K$}
    \Initialize {For all $i\in\N$, %$\xi_k=1$, \\
      set $\eta_{i}$ s.t. $\eta_{i} \neq \eta_{k} \;\forall \{i,k\}\in\E$ \\		
      $m_{i} \leftarrow 0$, \quad $\bar{h}_{i} \leftarrow \bar{h}^0$ }
    \State \textbf{For every $i\in\N$: }
    % \While{$\xi_k = 1$}
    \State \quad\textbf{Update} $\N_{\eta_i}:=\{k\in\N_i |
    \eta_{k} > \eta_{i} \}$
    \State\quad\textbf{If} {$|\N_{\eta_i}| \geq
      \bar{h}_{i}$ }
    \State\quad\quad\textbf{If} $\bar{h}_{i} = 6$
    \textbf{or} $m_{i} \leq \bar{m}$
    \State\quad\quad\quad $\eta_{i} =
    \max_{k\in\N_{\eta_i}}\eta_{k}+1$
    \State\quad\quad\quad $m_{i} \leftarrow m_i + 1$
    \State\quad\quad\quad Send $\eta_i$ to neighbors
    $\N_i$
    \State\quad\quad\textbf{else if } $m_{i} > \bar{m}$
    \textbf{ and } $\bar{h}_{i}<6$
    \State\quad\quad\quad $m_{i} \leftarrow 0$, \quad
    $\bar{h}_{i} \leftarrow\bar{h}_{i} + 1$ 
    \State\quad\quad\textbf{end}
    \State\quad\textbf{end} 
    % \EndWhile
    % \Outputs {optimality $\gamma_k$}
  \end{algorithmic}
  \label{algo:degree}
\end{algorithm}

We next explain the pseudocode of Algorithm~\ref{algo:degree}. Every
node $i$ starts with an initial number $\eta_i$ such that
$\eta_i\neq\eta_k$, for all $\{i,k\}\in\E$. These numbers induce an
acyclic orientation by declaring that node $i$ is a tail of $\{i,k\}$
if $\eta_i < \eta_k$, cf. Remark~\ref{rem:acyclic_graph}.
% The algorithm updates $\eta_i$, $i\in \N$, until no
% longer necessary, which results into a particular graph orientation
% and out-degree, $|\N_{\eta_i}|$, for each node.
Under the algorithm, every node $i$ % computes  its
% out-degree and
recursively updates $\eta_i$ if its current out-degree is bigger than
or equal to a number $\bar{h}_i$, initially set to $\bar{h}^0 = 2$,
for all~$i$. The update of $\eta_i$ follows a simple rule to choose a
number bigger than $\eta_k$, for all $k\in\N_i$. In this way, node $i$
becomes a sink node with out-degree zero in the new induced
orientation. Node $i$ then sends $\eta_i$ to all $k\in\N_i$ for them
to recompute their out-degree. If all nodes do not require any further
update in their $\eta$, Algorithm~\ref{algo:degree} converges to an
orientation in which all nodes have an out-degree of at most two. If,
instead, some nodes require an update for more than $\bar{m}>1$ times,
then the strategy has these nodes increase its $\bar{h}_i$ by one
(since an orientation with out-degree less than two might not
exist). Any node $i$ that again updates its variable $\bar{m}$ times
has $\bar{h}_i$ increased in a similar way. The procedure repeats
until every node stops updating.

The following result establishes the convergence properties of
Algorithm~\ref{algo:degree}.

\begin{proposition}\longthmtitle{Convergence of
    Algorithm~\ref{algo:degree}}
  \label{prop:outdegree}
  Algorithm~\ref{algo:degree} converges with $\bar{h}_i\leq 6$ for all
  $i\in\N$ in a finite number of iterations.
\end{proposition}
\begin{proof}
  Because $\bar{m}<\infty$, it is sufficient to show that every node
  stops updating $\eta_i$ in a finite number of steps if
  $\bar{h}_i=6$, $\forall i\in\N$. Since every simple planar graph has
  at least one node with degree strictly less than six, this node
  stops updating $\eta_i$ as $|\N_{\eta_i}|\leq 5 <\bar{h}_i$.  We
  then consider the subgraph of $\graph$ induced from
  $\N\setminus\{i\}$. This subgraph is also planar so we can find
  another node that will stop updating $\eta$ after a finite number of
  steps. The result follows by repeating this argument.
\end{proof}

With the orientation induced by the $\eta_i$ variables resulting from
the algorithm, one can find a coloring of the graph with $\bar{h}:=
\max_{i\in\N}\bar{h}_i$ colors.  Given the set of numbers $C:=\{1, 2,
\dots,\bar{h} \}$, Algorithm~\ref{algo:diameter},
from~\cite{SG-MHK:93}, assigns a number in~$C$ to each $i\in\N$.
% The proof of the convergence of
% Algorithm~\ref{algo:diameter} can be found in~\cite{SG-MHK:93}.
The resulting acyclic orientation has diameter at most $\bar{h}-1$ due
to~\eqref{eq:diameter_chrom}. 

\begin{algorithm}
  \caption{}
  \begin{algorithmic}[1]
    % \Inputs {System matrices: $A,B,K$}
    \Initialize {For all $k\in\N$, set $\zeta_k\in
      \{1,2,\dots,\bar{h}_k\}$ }
    \State \textbf{For every $k\in\N$, }
    \State\quad\textbf{Update} $\N_{\zeta_k}:=\{i\in\N_k | \zeta_i =
    \zeta_k \text{ and } \eta_i>\eta_k \}$
    \State\quad\quad\textbf{If}
    $\N_{\zeta_k}\neq \emptyset$ 
    \State\quad\quad\quad Choose $\zeta_k \in
    \{1,2,\dots,\bar{h}_k\}\setminus\{\zeta_i| i\in \N_{\zeta_k}\}$
    \State\quad\quad\textbf{end}
    \State \quad \textbf{Send} $\zeta_k$
    to $i\in\N_k$
    % \Outputs {optimality $\gamma_k$}
  \end{algorithmic}
  \label{algo:diameter}
\end{algorithm}
% Notice that any node $i$ with degree one never change its $\eta_i$
%
% \marginJC{Be precise! $\eta$ does not exist by itself. You mean that
%   a node $k$ with degree 1 never changes its $\eta_k$, no? }
%
% in Algorithm~\ref{algo:degree} and the corresponding $\bar{h}_i$
% remains at three. Define $\D_3 = \{i\in\N|\bar{h}_i = 3\}$.
% It is straightforward to see that the vertex-induced subgraph of
% $\graph$, $\graph[\D_3]$, has all the paths with length at most
% two. Since $\D_3$ contains many nodes in a typical electrical
% network due to Assumption~\ref{assum:exp_deg_dist}, and the small
% average degree of the network many paths in the original network
% have length bounded by two. We can further define $\D_4 =
% \{i\in\N|\bar{h}_i = 4\}$ and state that all the paths in
% $\graph[\D_4]$ have length at most three. Similar arguments apply
% for $\D_5 = \{i\in\N|\bar{h}_i = 5\}$. The result is sharper than
% directly implementing the algorithm in~\cite{SG-MHK:93} which only
% guarantees that every path has a length at most five.
%
	
In Algorithm~\ref{algo:degree}, setting $\bar{m}$ too small can be
conservative because the strategy gives up too early in finding
orientations with small diameter by rapidly increasing $\bar{h}_k$
toward six. On the other hand, setting $\bar{m}$ too large slows down
the convergence. The challenge lies then in characterizing the value
$\bar{m}$ that strikes a balance between maximizing the convergence
rate and minimizing the average path length of the resulting
orientation.  The optimal choice of $\bar{m}$ depends on the size of
the network and the constant $\lambda$ in the geometric degree
distribution. Theorem~\ref{thm:subgraph_degree} illustrates how
$\bar{m}$ is related to these factors.

\begin{theorem}\longthmtitle{On the degree of subgraphs and
    convergence of
    Algorithm~\ref{algo:degree}}\label{thm:subgraph_degree}
  If $\bar{h}_i^0 = c_0 <6$, for all $i\in\N$ and $\bar{m} = \infty$,
  then Algorithm~\ref{algo:degree} converges if and only if every
  vertex-induced subgraph of $\graph$ has at least one node with
  degree less than $c_0$.
\end{theorem}
\begin{proof}
  We first show the implication from right to left.  We term every
  node $k$ with degree less than $c_0$ ``stable'', because it will not
  update its $\eta_k$ regardless of the change of $\eta$ of any other
  node. Let $\St_1$ be the set of stable nodes of $\graph$. We then
  consider the induced subgraph $\graph[\St_2]$, where
  $\St_2=\N\setminus\St_1$. The graph $\graph[\St_2]$ also has at
  least one node with degree less than $c_0$. Again, nodes in
  $\graph[\St_2]$ with degree less than $c_0$ are called ``stable''
  because they will not change their $\eta$ in response to the change
  of $\eta$ of any other node in $\St_2$. Every stable node $i$ in
  $\graph[\St_2]$ updates $\eta_i$ at most once. The reason of the
  single update is the following. First, a stable node $i$ in $\St_2$
  is connected to at least one node in $\St_1$. Otherwise, node $i$ is
  in $\St_1$.
  % Recall that all the nodes in $\St_1$ never change $\eta$.
  Hence, once a stable node $i\in\St_2$ updates $\eta_i$, its
  out-degree, $|\N_{\eta_i}|$, is at most $\bar{h}_i-1$ and the node
  will not update $\eta_i$ again. We can reason with $\St_3, \St_4,
  \cdots, \St_s$ in a similar way until $\cup_{\alpha
    =1,2,\cdots,s}\St_\alpha = \N$ (note that $s<\infty$ because
  $N<\infty$) % Repeat the arguments of the stable nodes for all
  % $\St_i$, $i = 1,2,\cdots,s$.
  and conclude that Algorithm~\ref{algo:degree} converges in finite
  time.
  
  Next, we show the implication from left to right. If there exists a
  vertex-induced subgraph with all nodes having degree at least $c_0$,
  then we can show that at least one node in the subgraph updates
  $\eta$ infinitely often. Recall that every acyclic orientation has a
  source node. The source of the subgraph updates its $\eta_i$ because
  $|\N_{\eta_i}|\geq c_0$. After this update, the orientation of the
  subgraph remains acyclic because it is induced by the values of the
  variables~$\eta_i$. % with every pair of connected nodes have different
  % $\eta$.  
  As a consequence, at least one node of the subgraph becomes a source
  and makes an update. The sequence repeats for infinite times because
  there always exists one node in the subgraph with $|\N_{\eta_i}| \ge
  c_0$.  Since the number of nodes is finite, there exists at least
  one node that updates $\eta$ infinitely often. As a consequence,
  Algorithm~\ref{algo:degree} does not converge.
\end{proof}

Theorem~\ref{thm:subgraph_degree} provides insight into the selection
of $\bar{m}$ in Algorithm~\ref{algo:degree}.  With the notation of the
proof, if we set $\bar{m} = \infty$ and Algorithm~\ref{algo:degree}
converges with $\bar{h}^0_i = c_0 <6$, for all $i\in\N$, the nodes in
$\St_\alpha$ update their variables at least the number of times that
nodes in $\St_\beta$ do, for $\beta<\alpha$, because $\St_\alpha$
becomes ``stable'' after $\St_\beta$ does.  If, instead, we choose a
finite $\bar{m}$, then only a subset of nodes of $\N$ keep their
variable $\bar{h}_i$ non-increasing, while the remaining nodes are
forced to increase it, resulting in a more conservative upper bound of
their out-degrees. The number of updates required for the last node
being stable highly depends on the network topology. If $\lambda$ is
large, then most nodes have degree one or two,
%
%\marginJC{Again, i don't understand this statement without knowing what $d_0$ is}%\marginch{clarified earlier}
%
and $\St_1$ contains most nodes in $\N$. In this case, we can expect
$\cup_{\alpha=1,\cdots,s}\St_\alpha = \N$ with a small $s$, so
selecting a small $\bar{m}$ still provides a small $\bar{h}$ with fast
convergence time.
% The subgraph induced from $\N\setminus\St_1$ is likely to be
% disconnected because they are cut by $\St_1$.  Each disconnected
% subgraph has at least one node being stable if
% Algorithm~\ref{algo:degree} converges with the given $\bar{h}_i$.
% Since each disconnected subgraph only has the number of nodes much
% smaller to $|\St_1|$, the stable nodes can propagate to the entire
% subgraph with only small $s$.
Our experience shows that selecting $\bar{m}$ around $10$ is
sufficient to yield an optimal coloring for electrical networks. The
number may increase for some large networks that involve thousands of
nodes.
	
\section{Simulations}\label{sec:Simulation}
Here, we validate the performance of the scheduled-asynchronous
algorithm over the six bus test cases in~\cite{AJW-BW:96}, IEEE 14,
30, and 57 bus test cases.  We first use
Algorithms~\ref{algo:degree}-\ref{algo:diameter} to find an acyclic
orientation of each test case, cf. Table~\ref{tab:ori}.
% Since the network topology
% is relatively static compared to other uncertainties in electrical
% grids such as load uncertainties, Algorithms~\ref{algo:degree} and
% \ref{algo:diameter} can be implemented over a separated time scale
% from that of the scheduled-asynchronous algorithm. Finding an
% orientation before implementing Algorithm~\ref{algo:pd} is therefore
% reasonable.
\begin{table}[h]
  \begin{center}
    \caption{Simulation parameters and results of Algorithm~\ref{algo:degree}.}
    \begin{tabular}{ || c| c | c| c | c  || }
      \hline
      & $\bar{m}$ & $\bar{h}^0$ & Final $\bar{h}$ & \begin{tabular}{@{}c@{}}Diam. of \\ acyc. ori. \end{tabular}  \\
      \hline
      6 bus  & 10 & 2 & 4 & 3  \\ \hline
      14 bus  & 10 & 2 & 3 &  2  \\ \hline
      30 bus   & 10 & 2 & 3  &  2 \\ \hline
      57 bus   & 10 & 2 & 3  &  2 \\ \hline
    \end{tabular}
    \label{tab:ori}
  \end{center}
\end{table}
This results in orientations with diameter two for all test cases
except the six bus test cases, with diameter three.

% We observe that the number of iterations per node grows only
% linearly with respected to the number of edges in the electrical
% network. The property shows a sub-linear rate of convergence.

\begin{figure*}
  \centering 
  \subfigure[Six bus test
  case]{\includegraphics[width=.49\linewidth]{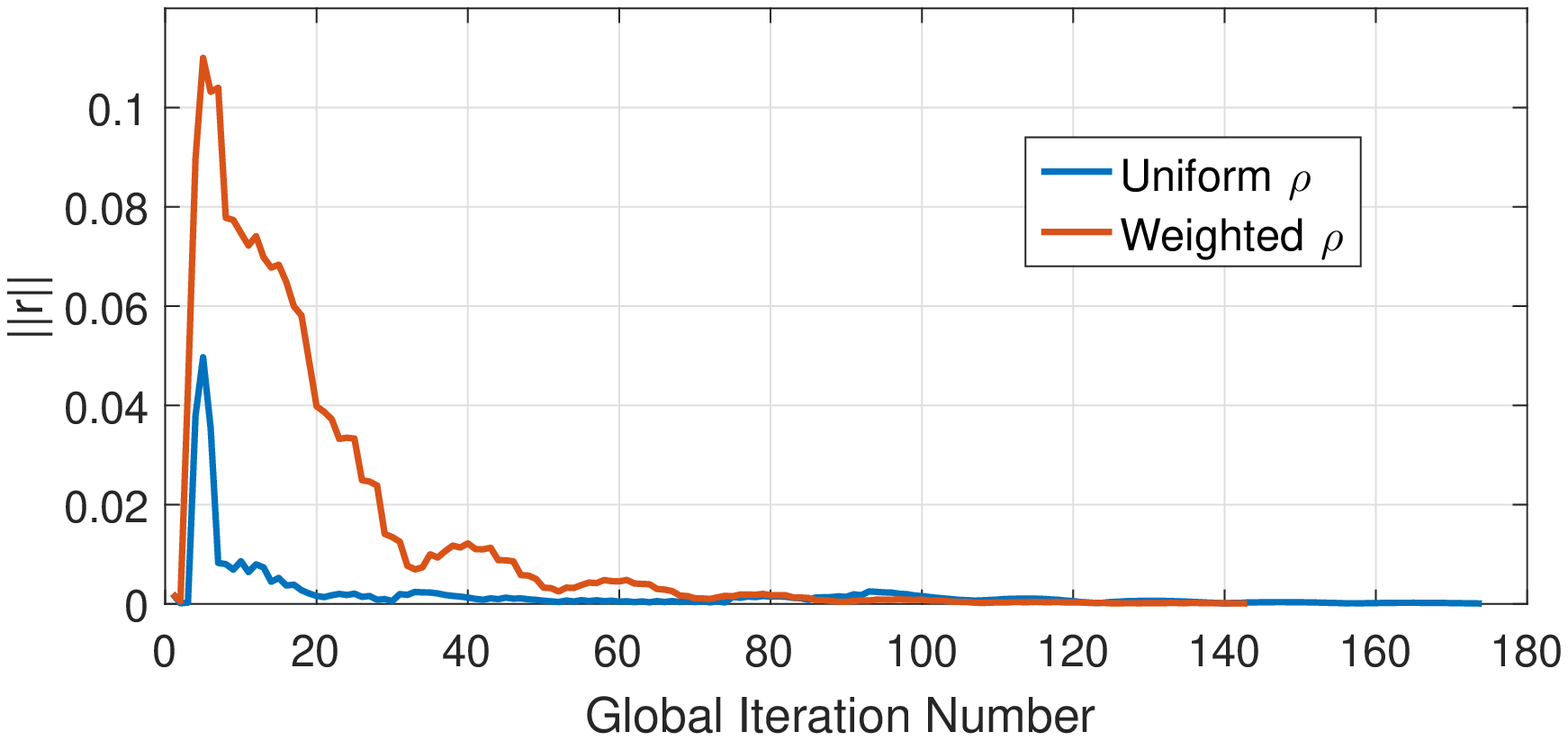}}\label{fig:6bus}
  \subfigure[IEEE 14 bus test
  case]{\includegraphics[width=.49\linewidth]{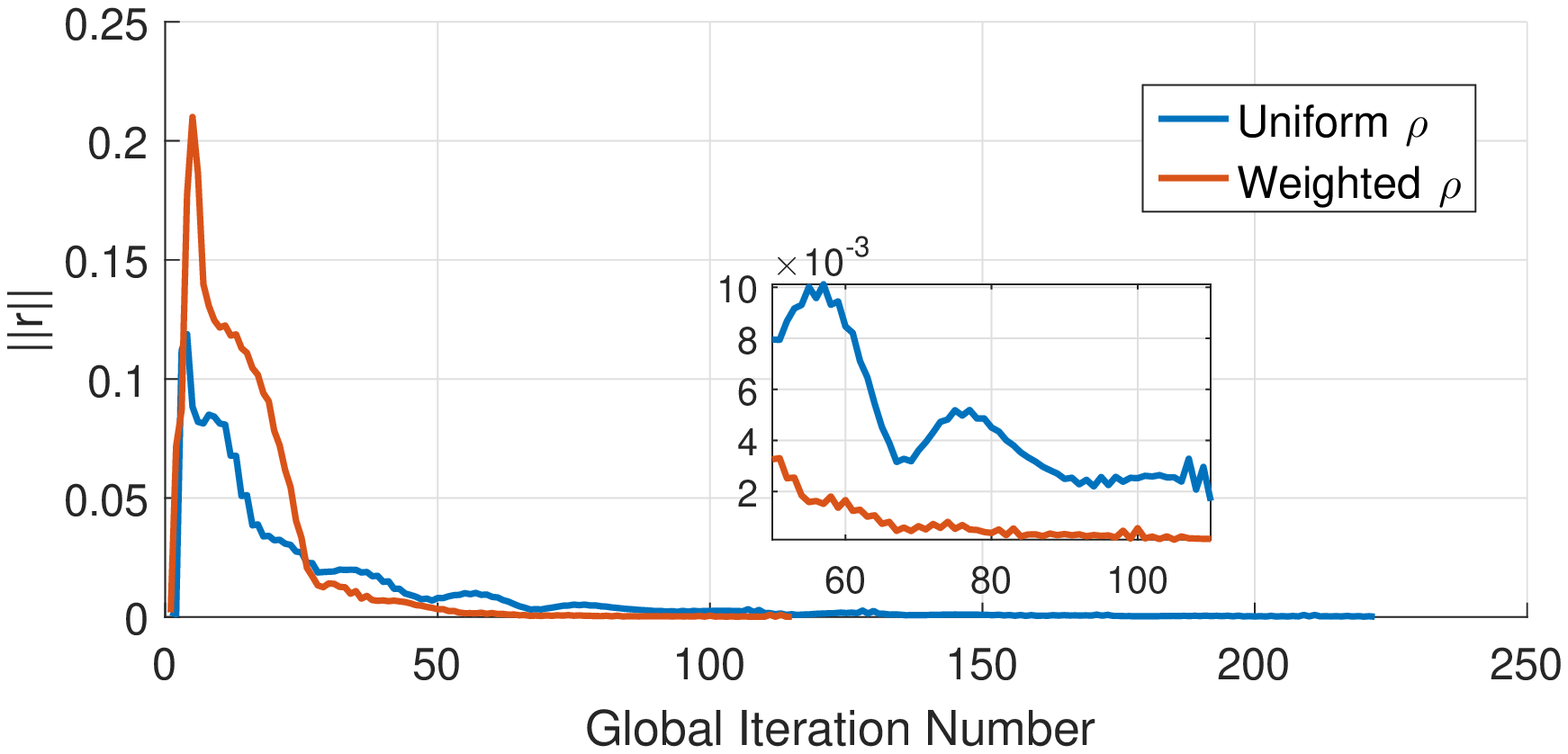}}\label{fig:14bus}
  \\
  \subfigure[IEEE 30 bus test
  case]{\includegraphics[width=.49\linewidth]{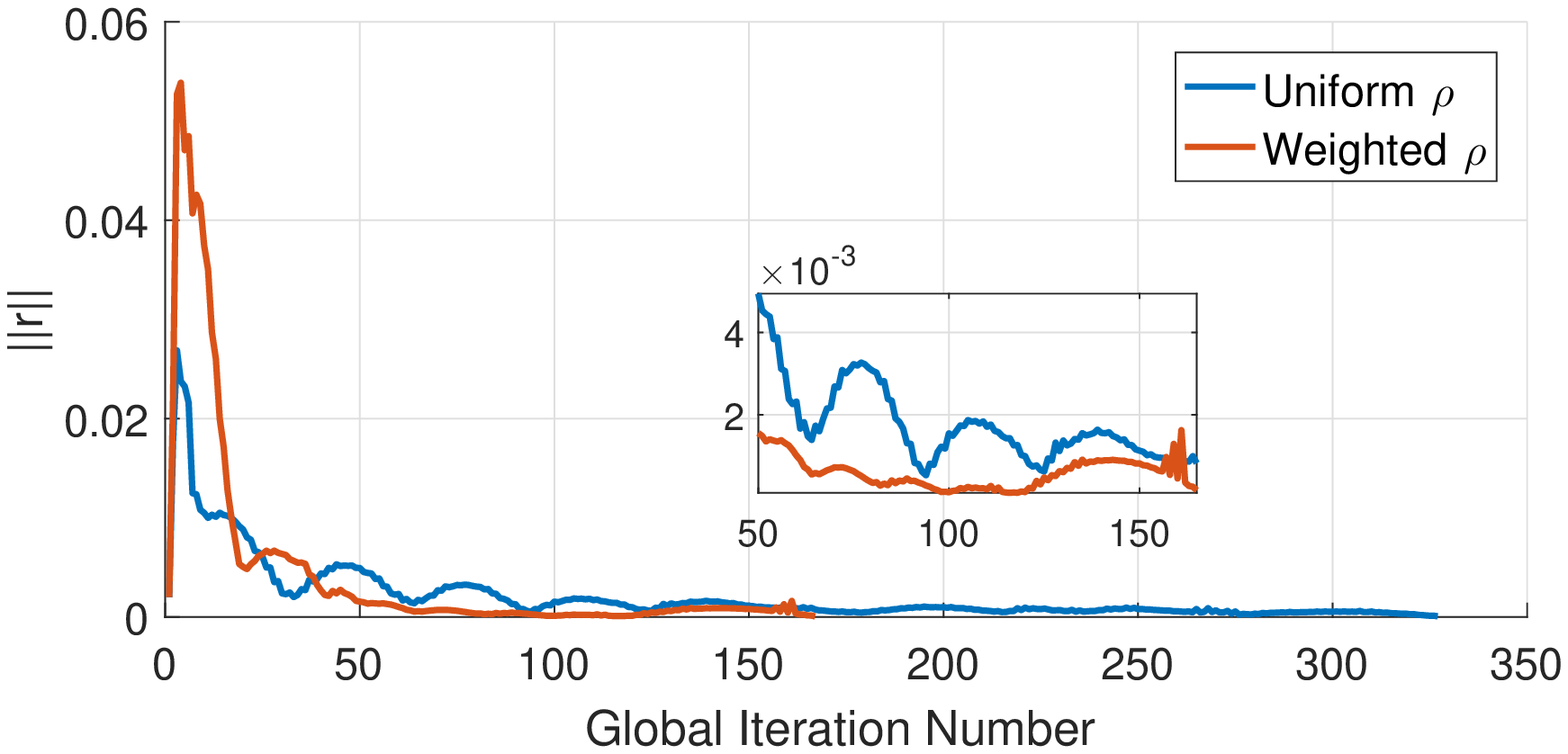}\label{fig:30bus}}
  \subfigure[IEEE 57 bus test
  case]{\includegraphics[width=.49\linewidth]{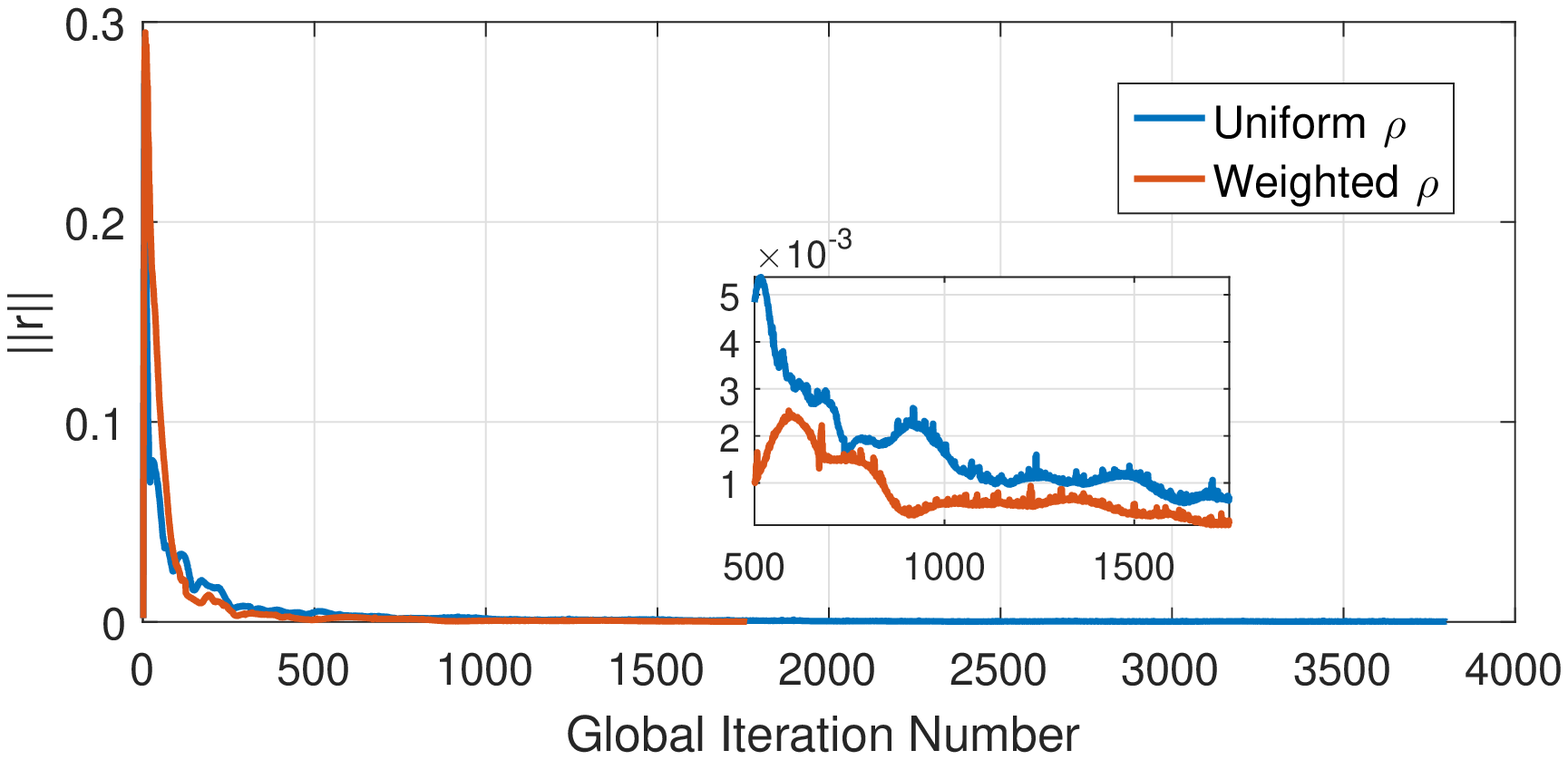}}\label{fig:57bus}
  \caption{Convergence of the scheduled-asynchronous algorithm for
    various test cases: (a) six bus test case
    in~\cite{AJW-BW:96}, (b) IEEE 14, (c) IEEE 30, and (d) IEEE
    57.}\label{fig:algorithm-1-buses}
\end{figure*}

Figure~\ref{fig:algorithm-1-buses} shows the convergence of
Algorithm~\ref{algo:pd} for the various test cases.  The stopping
criteria is $\gamma_{i}<10^{-4}$ for all $i\in\N$.  The horizontal
axis in the plots is the global iteration number, not the iteration
per bus (the number of iterations per bus is roughly the global
iteration number divided by the diameter of the acyclic orientation).
\begin{table}[h]
  \begin{center}
    \caption{Number of iterations needed for $\gamma_{i}<10^{-4}$,
      $\forall i\in\N$.}
    \begin{tabular}{ || c| c | c| c | c || }
      \hline
      & Iter./bus & \begin{tabular}{@{}c@{}}Iter./bus \\pack. drop.\end{tabular} & \begin{tabular}{@{}c@{}}Iter./bus \\ weighted $\rho$\end{tabular} & $\rho_0$  \\
      \hline
      6 bus  & 62 & 65 & 50 & 700 \\ \hline
      14 bus  & 110 & 127 & 57   &  700 \\ \hline
      30 bus   & 140 & 260 & 82  &  700 \\ \hline
      57 bus   & 1520 & 1810 & 660  & 1000 \\ \hline
    \end{tabular}
    \label{tab:sim}
  \end{center}
\end{table}
Table~\ref{tab:sim} and Figure~\ref{fig:algorithm-1-buses} show that
the weighted selection of $\rho$, discussed in
Remark~\ref{rem:para_selection}, leads to a much faster convergence
than the uniform $\rho_{ik} = \rho_0$, for all $\{i,k\}\in\E$. We have
also simulated the case with unreliable communication, where every
link has a $10\%$ probability of packet drop if the previous
communication was successful. Table~\ref{tab:sim} shows that the
Algorithm~\ref{algo:pd} still converges, albeit requiring more
iterations than the case with no packet
drops. Figure~\ref{fig:algoritm-1-drops} illustrates the transient
behavior for the packet drop case.

\begin{figure*}
  \centering
  \subfigure[IEEE 14 bus test case]{
    \includegraphics[width=.49\linewidth]{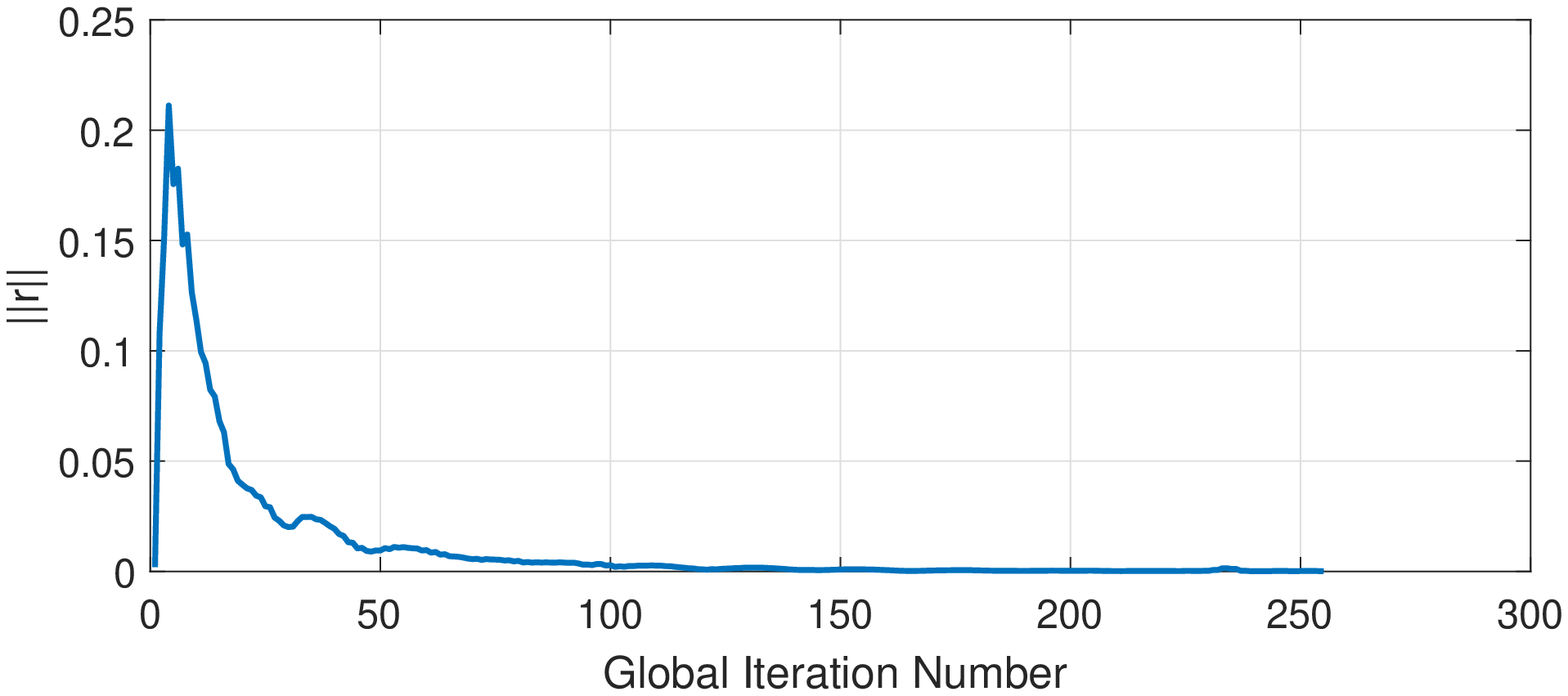}}\label{fig:14bus_drop}
  \subfigure[IEEE 30 bus test case]{
    \includegraphics[width=.49\linewidth]{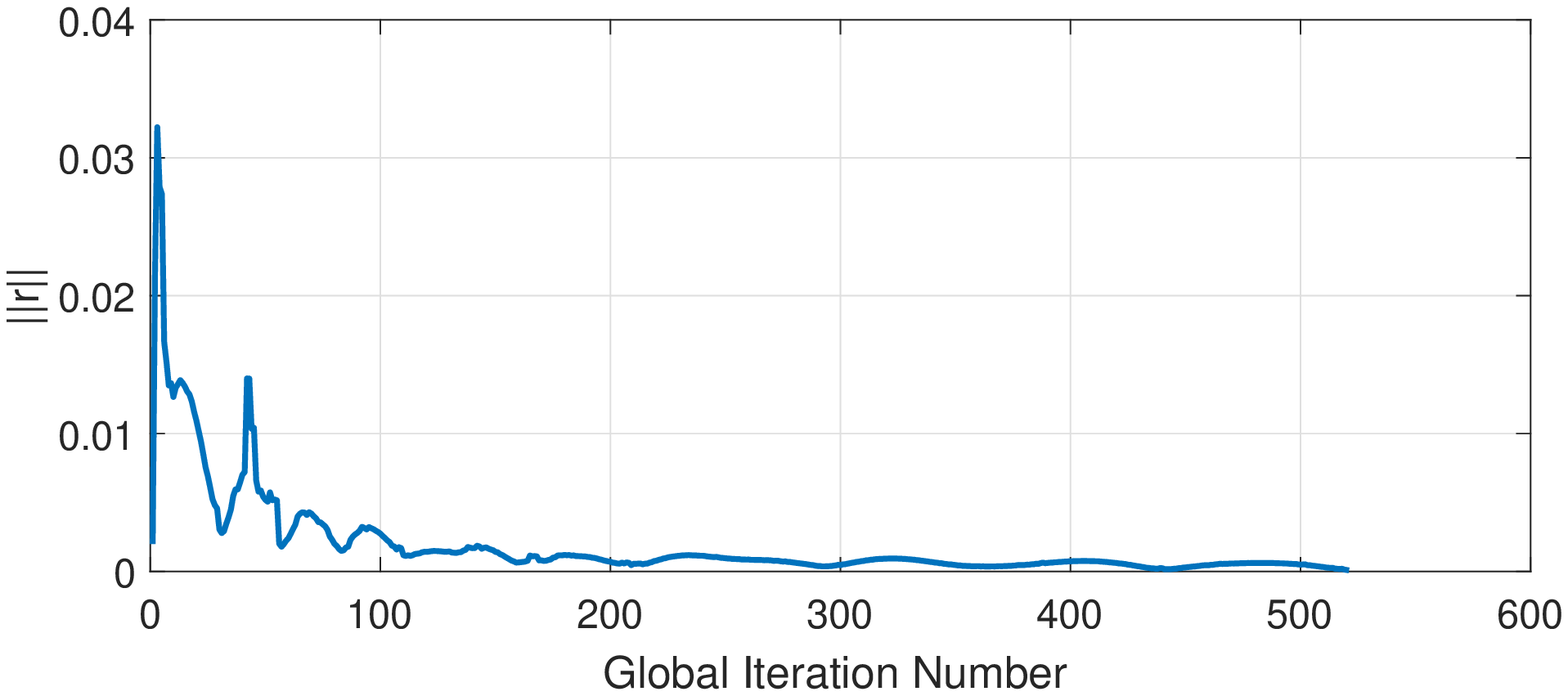}}\label{fig:30bus_drop}
  \caption{Convergence of the scheduled-asynchronous algorithm under
    packet drops.}\label{fig:algoritm-1-drops}
\end{figure*}

Finally, we evaluate the impact of the acyclic orientation in the
algorithm convergence in Figure~\ref{fig:algorithm-1-acyclic}.  We
simulate how Algorithm~\ref{algo:pd} converges for the IEEE 14 and 30
bus test cases with an arbitrarily chosen acyclic orientation
instead of the one obtained from the execution of
Algorithms~\ref{algo:degree}-\ref{algo:diameter}. Compared with
Figure~\ref{fig:algorithm-1-buses}(b)-(c), one can observe that the
algorithm requires many more iterations to converge and that $\|r\|$
increases dramatically at several iterations. These abrupt changes can
also be explained as the result of several long paths in the
digraph. In fact, recall that every node in a path makes one iteration
before transferring its update to the following node.  Starting with
one of the terminal nodes in a long path, the update propagates
through many nodes before reaching the other terminal node for its
next update, at which point it may introduce a big change on its
decision variables resulting in a dramatic change on
$\|r\|$. Therefore, the selection of an acyclic orientation with small
diameter has the added benefit, beyond reducing the average waiting
time per iteration, of resulting in less abrupt changes in the
algorithm execution.
	
\begin{figure*}
  \centering
  \subfigure[IEEE 14 bus test case]{
    \includegraphics[width=.49\linewidth]{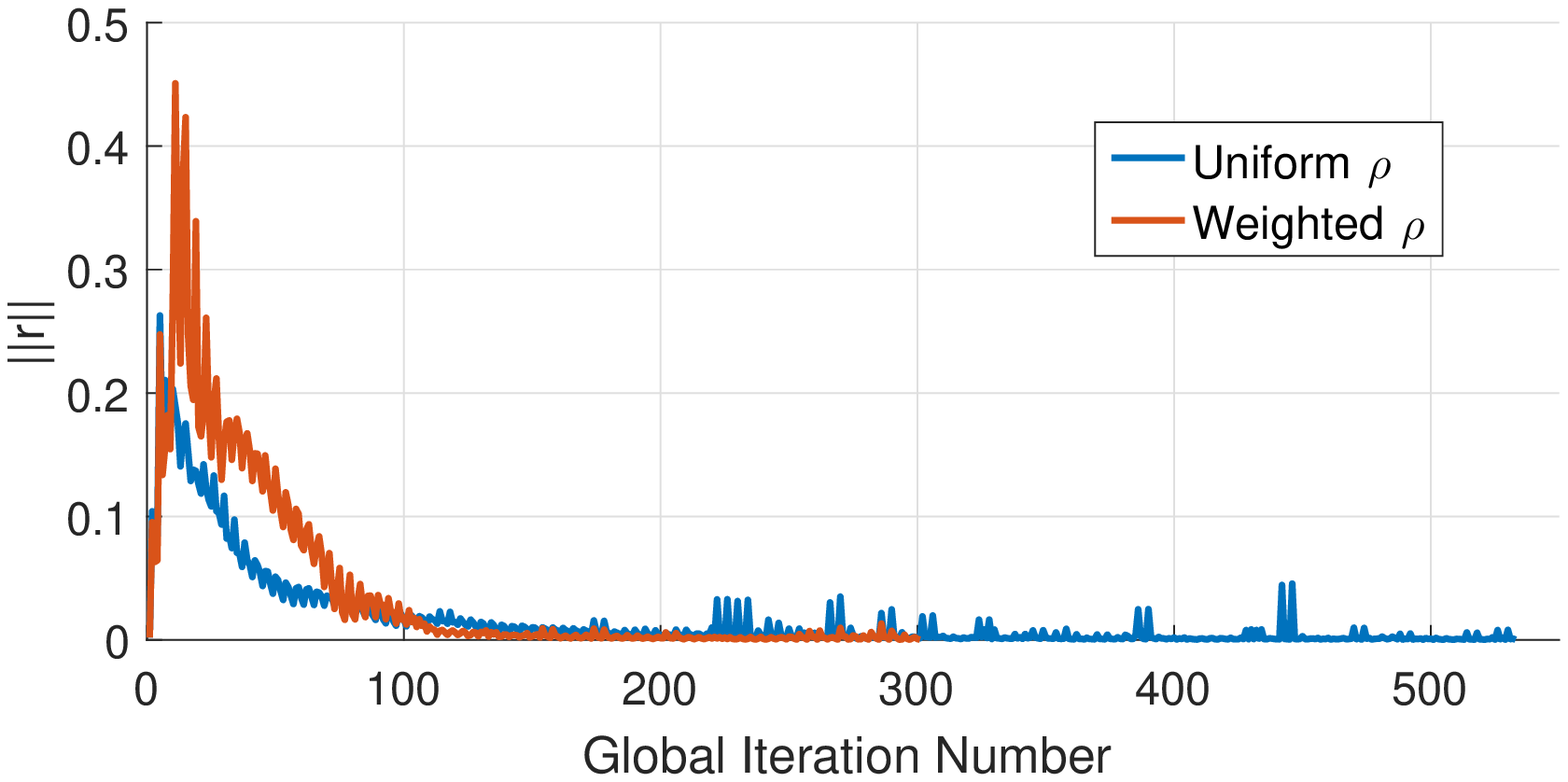}}\label{fig:14bus_wo_ori} 
  \subfigure[IEEE 30 bus test case]{
    \includegraphics[width=.49\linewidth]{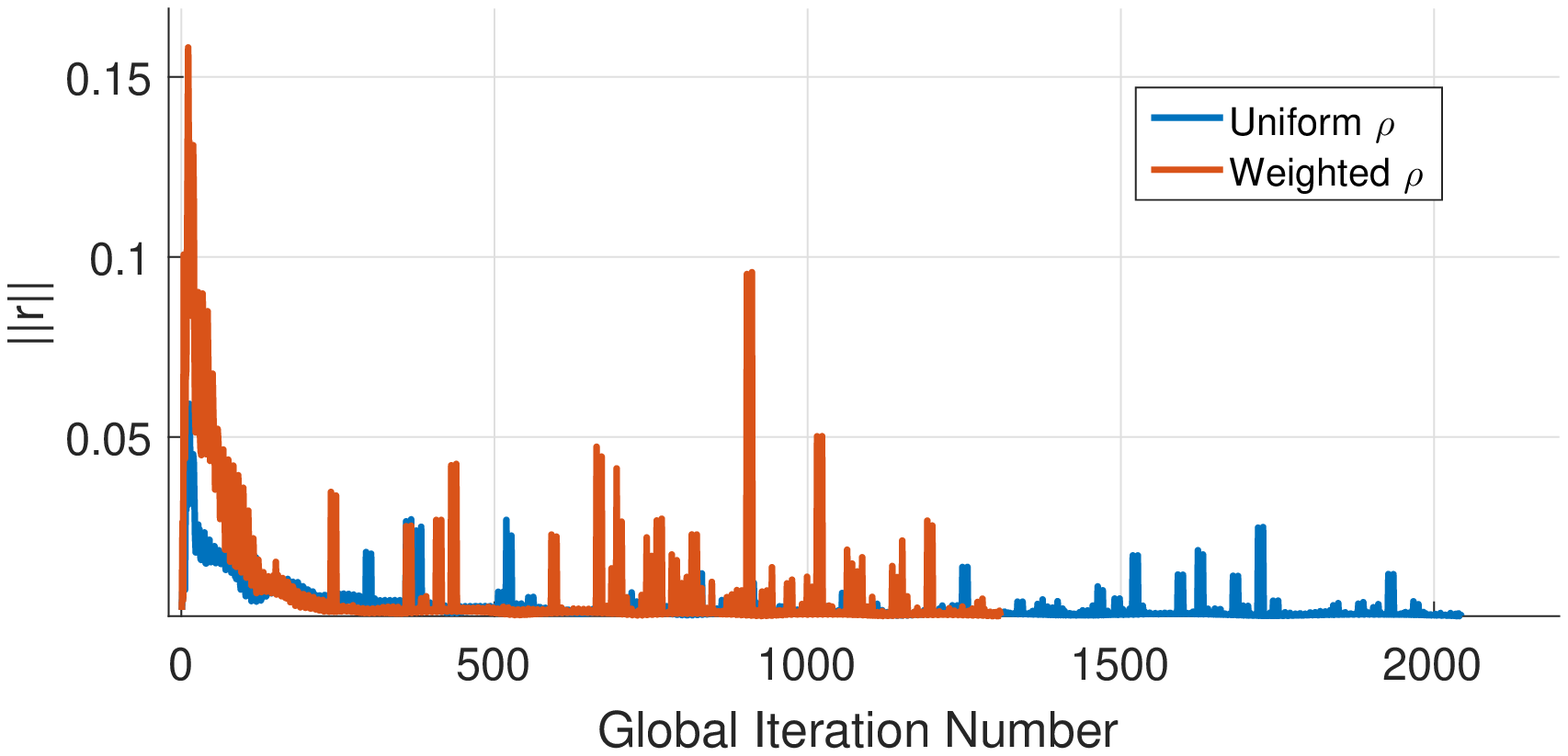}}\label{fig:30bus_wo_ori}
  \caption{Convergence of the scheduled-asynchronous algorithm with
    arbitrary acyclic orientations of the network
    graph.}\label{fig:algorithm-1-acyclic}
\end{figure*}

% To test the algorithm robustness, we simulate the operation scenario
% in which, if one node waits for more than a certain threshold time
% to hear from its neighbors, it precedes with its optimization by
% using the state of the previous step of neighboring nodes. The $5\%$
% and $10\%$ of unsuccessful line communications are simulated. The
% algorithm remains convergent at the expense of an increase in the
% number of iterations per node. This increase in the number of
% iterations can be considered the price to pay to bring the variables
% back to optimality under this type of disturbance.

\section{Conclusions}\label{sec:conclusion}
We have designed the scheduled-asynchronous algorithm to solve SDP
convexified OPF problems in a distributed way.  Under the proposed
strategy, every pair of nodes connected in the electrical network
update their local variables in an alternating fashion and the
ordering of node updates is encoded by an orientation of the network.
We have established the algorithm convergence to the optimizer when
the orientation is acyclic and shown how, when the network is
bipartite, the strategy corresponds to the ADMM scheme and has
therefore $O(1/n)$ convergence rate.  The convergence result does not
require strong convexity of the cost function, which makes it
especially suitable in OPF applications. To improve the algorithm
convergence rate, we have introduced a distributed graph coloring
algorithm that finds an acyclic orientation with small diameter for
networks with geometric degree distribution.  Future work will explore
the characterization of the convergence rate for general network
topologies, the optimal selection of the algorithm parameters, and the
formal analysis of the algorithm robustness properties observed in
simulation.
	
\section*{Acknowledgments}
This research was supported by the ARPA-e Network Optimized
Distributed Energy Systems (NODES) program, Cooperative Agreement
DE-AR0000695.

\bibliographystyle{IEEEtran}
\bibliography{JC,Main-add,Main,alias}

\appendix
\renewcommand{\theequation}{A.\arabic{equation}}
\renewcommand{\thetheorem}{A.\arabic{theorem}}

\section{Appendices}
\subsection{Auxiliary Results for the Proof of
  Theorem~\ref{thm:converge}}
We gather here various auxiliary results used in the proof of
Theorem~\ref{thm:converge}.  Our first result shows that two classes
of convex optimization problems with separable cost functions have the
same optimal solution.

\begin{lemma}\label{lem:equi_opt}
  Let $\phi, \varphi:\mathcal{H}^{n} \rightarrow \mathbb{R}$ be convex
  and differentiable, and let $\X$ be a convex set. Then
  \begin{subequations}\label{eq:eq_opts}
    \begin{alignat}{2}\label{eq:eq_opts1}
      &X^\star \in \argmin_{X\in\X}\phi(X) + \varphi(X),\\
      \label{eq:eq_opts2} \hspace{-5mm}\Leftrightarrow \; &X^\star \in
      \argmin_{X\in\X}\phi(X) + \langle \bigtriangledown
      \varphi(X^\star), X \rangle.
    \end{alignat}
  \end{subequations}
\end{lemma}
\begin{proof}
  The necessary and sufficient condition for $X^\star$ being the
  optimal solution of optimization~\eqref{eq:eq_opts1} is that
  \begin{align*}
    &\big\langle \bigtriangledown \big( \phi(X^\star) +
    \varphi(X^\star)\big),X-X^\star \big\rangle \geq 0,
  \end{align*}  
  for all $X \in\X$. We can rewrite the condition above as 
  \begin{align}\label{eq:sep_opt_cond}
    & \big\langle \bigtriangledown \phi(X^\star),X-X^\star \big\rangle
    +\big\langle \bigtriangledown
    \varphi(X^\star),X-X^\star\big\rangle \geq 0,
  \end{align}
  for all $X \in\X$. Eq.~\eqref{eq:sep_opt_cond} is also the
  optimality condition of optimization~Eq.~\eqref{eq:eq_opts2}, which
  completes the proof.
\end{proof}

We use Lemma~\ref{lem:equi_opt} to establish 
% The interchangeable optimality conditions of the two optimizations
% are useful for the proof of Lemma~\ref{lem:pre_proof}.
% Lemmas~\ref{lem:pre_proof} through~\ref{lem:invariant_set} are
% particularly established for the proof of
% Theorem~\ref{thm:converge}. All the lemmas assume that conditions
% 1)-4) in Theorem~\ref{thm:converge} hold. Lemma~\ref{lem:pre_proof}
two inequalities that will be employed to show that the
function~\eqref{eq:LaSalle} is non-increasing along the algorithm
executions. The notation we employ next is carried over from the proof
of Theorem~\ref{thm:converge}.

\begin{lemma}\label{lem:pre_proof}
  Under the assumptions of Theorem~\ref{thm:converge}, the following
  inequalities hold
  \begin{subequations}
    \label{eq:pre_proof}
    \begin{alignat}{2}
      & \sum_{i=1}^N \Big( f_i(W^\star_i) - f_i(W_i^{n+1}) \Big) \leq
      {p^{\star}}^\top r^{n+1}, \label{eq:prove_1} \displaybreak[0]
      \\
      & \sum_{i=1}^N \Big( f_i(W_i^{n+1}) - f_i(W^\star_i) \Big) \leq
      -{p^{n+1}}^\top r^{n+1} \label{eq:prove_2}
      \\
      \nonumber &+ \hspace{-2mm}\sum_{\{i,k\}\in\hat{\E}}\hspace{-2mm}
      \rho_{ik} D_{ki}^\top(W_k^{n} - W_k^{n+1})D_{ik}(W_i^\star -
      W_i^{n+1}) .
    \end{alignat}
  \end{subequations}
\end{lemma}
\begin{proof}
  \textbf{Eq.~(\ref{eq:prove_1}).}  Since \textbf{(P2)} is convex and
  Slater's condition holds, strong duality follows and the KKT
  conditions are necessary and sufficient for the optimal solution
  of~\eqref{eq:OPF_algo}. Strong duality together with the KKT
  conditions imply
  \begin{align}\label{eq:opt}
    W_{\N}^\star = \underset{W_i\in\W_i}{\argmin}\sum_{i\in\N}
    f_i(W_i)+ \hspace{-2.5mm}
    \sum_{\{l,k\}\in\hat{\E}}\hspace{-2.5mm}{p_{lk}^{\star}}^\top
    G_{lk}(W_{l},W_k).
  \end{align}
  Since $W_{\N}^\star$ is the optimizer, using this inequality we
  deduce
  \begin{align*}%\label{eq:prove_ineq}
    &\sum_{i=1}^{N} f_i(W^\star_i)+\hspace{-3mm}
    \sum_{\{l,k\}\in\hat{\E}}{p_{lk}^{\star}}^\top r_{lk}^\star
    \\ \nonumber
    &\hspace{14mm}\leq \sum_{i=1}^{N} f_i(W_i^{n+1}) +
    \sum_{\{l,k\}\in\hat{\E}}{p_{lk}^{\star}}^\top r_{lk}^{n+1}.
    % \\
    % &\Longrightarrow \sum_{i\in\N} f_i(W^\star_i) - \sum_{i\in\N}
    % f_i(W_i^{n+1}) \leq (p^{\star})^\top r^{n+1}.
  \end{align*}
  Inequality~\eqref{eq:prove_1} follows by noting that $r_{lk}^\star =
  0$, for all $\{l,k\}\in\hat{\E}$.
  
  \textbf{Eq.~(\ref{eq:prove_2}).}  We start by
  rewriting~\eqref{eq:local_update} with the number of iterations
  instead of the time index,
  \begin{align}\label{eq:local_update_iter}
    &W_i^{n+1} = \argmin_{W_i\in\W_i}f_i(W_i)
    \\
    \nonumber &\;\;+\hspace{-2mm}\sum_{\{i,k\}\in
      \hat{\E}}\hspace{-2mm}\Big( {p_{ik}^n}^\top G_{ik}(W_i,W_k^n) +
    \frac{\rho_{ik}}{2}\|G_{ik}(W_i,W_k^n)\|^2 \Big) \\ \nonumber
    &\;\;+\hspace{-2mm}\sum_{\{k,i\}\in \hat{\E}}\hspace{-2mm}\Big(
    {p_{ik}^n}^\top G_{ki}(W_k^{n+1},W_i) +
    \frac{\rho_{ik}}{2}\|G_{ki}(W_k^{n+1},W_i)\|^2 \Big).
  \end{align}
  The superscript of every variable in~\eqref{eq:local_update_iter}
  represents the number of updates of the associated variable.  We
  resort to Lemma~\ref{lem:equi_opt} to
  rewrite~\eqref{eq:local_update_iter}, viewing $f_i$ as $\phi$ and
  grouping all the other objective functions as $\varphi$. In
  addition, $W_i^{n+1}$ and $W_i$ play the role of $X^\star$ and $X$
  in~\eqref{eq:eq_opts}, respectively. We then have
  \begin{align*}%\label{eq:update_i_eq}
    &W_i^{n+1} = \argmin_{W_i\in\W_i}f_i(W_i)
    \\
    \nonumber &\;\;+ \sum_{\{i,k\}\in \hat{\E}}\hspace{-1mm}\Big(
    {p_{ik}^{n}}^\top\hspace{-1mm}+ \rho_{ik}
    G_{ik}^\top(W_i^{n+1},W_k^{n})\Big) D_{ik}(W_i)
    \\
    \nonumber &\;\;+ \sum_{\{k,i\}\in \hat{\E}}\hspace{-1mm}\Big(
    {p_{ik}^{n}}^\top\hspace{-1mm}+ \rho_{ik}
    G_{ki}^\top(W_k^{n+1},W_i^{n+1})\Big) D_{ik}(W_i).
  \end{align*}
  With a slight abuse of notation, we denote by
  $G_{ki}^\top(\cdot,\cdot) \equiv G_{ki}(\cdot,\cdot)^\top$, the
  transpose of $G_{ki}(\cdot,\cdot)$ (similarly for the variables
  $D_{ki}^\top$).
  % \margins{just checking: you changed the notation of
  % $G_{ik}(W_i,W_k^n)$ by $G_{ik}^\top(W_i,W_k^n)$. Since
  % $G_{ik}^\top(W_i,W_k^n)$ gives you a matrix, I think we should
  % use the notation $G_{ik}(W_i,W_k^n)^\top$ (propagate this
  % change) } \margins{is there a $-W_i^{n+1}$ missing in the
  % previous equation (multiplying
  % ${\bigtriangledown}_{i}G_{ik}$)?}\marginch{Yes, I
  % dropped it on purpose because it doesn't change the
  % optimal solution.}
  % where ${\bigtriangledown}_{W_i}
  % G_{ik}:\mathcal{H}^{|\N_i|}\mapsto \mathbb{R}^4$
  % denotes the gradient operator.
  % \margins{we had a discussion about this already, and you say it
  % is not going to affect the solution, but can we mention that
  % we have dropped the terms $W_i^{n+1}$ and why not including
  % these terms does not matter?(from what it would be the
  % equivalent to equation A.25b)? the computations will be easier
  % to reproduce by a reader this way. Else, we can could directly
  % change equation (A.25b) into what we really use
  % here.}  %\marginch{I changed the statement of Lemma A1}
  According to this equation, evaluating the objective function at
  $W_i^\star$ gives rise to a larger value than at $% W_{i}\neq
  W_{i}^{n+1}$, and therefore,
  % is the optimizer in into the cost function of the
  % equation above gives a bigger value than the one from
  % $W_{i}^{n+1}$.
  % \margins{why don't we recall that $p_{ik}^{n+1} =
  % p_{ik}^n + \rho_{ik}G_{ki}(W_k^{n+1},W_i^{n+1}) $?
  % or refer to an equation for the update of $p$?}
  % \marginch{added after the next equation}
  \begin{align}
    &\label{eq:update_eq} \sum_{i=1}^N f_i(W_i^{n+1}) - f_i(W^\star_i)
    \leq \sum_{i=1}^N\bigg(\\ \nonumber &\sum_{\{i,k\}\in
      \hat{\E}}\hspace{-2.5mm}\Big( {p_{ik}^{n}}^\top\hspace{-2.5mm} +
    \rho_{ik} G_{ik}^\top(W_i^{n+1},W_k^{n})\Big)D_{ik}( W_i^\star -
    W_i^{n+1}) \\ \nonumber &+\hspace{-3mm}\sum_{\{k,i\}\in
      \hat{\E}}\hspace{-2.5mm} {p_{ik}^{n+1}}^\top D_{ik}( W_i^\star -
    W_i^{n+1})\hspace{-1mm}\bigg).
  \end{align}
  Note that we used~\eqref{eq:lagran_update} to substitute $p_{ik}^n +
  \rho_{ik}G_{ki}(W_k^{n+1}\hspace{-2mm},W_i^{n+1})$ by
  $p_{ik}^{n+1}$.  The term ${p_{ik}^{n}}^\top\hspace{-2.5mm} +
  \rho_{ik}G_{ik}^\top(W_i^{n+1},W_k^{n})$ in~\eqref{eq:update_eq} can
  be written as
  \begin{align*}
    &{p_{ik}^{n}}^\top\hspace{-2.5mm} + \rho_{ik}
    G_{ik}^\top(W_i^{n+1},W_k^{n}) \\
    = &{p_{ik}^{n}}^\top\hspace{-2.5mm} + \rho_{ik}\Big(
    D_{ik}^\top(W_i^{n+1}) + D_{ki}^\top(W_k^{n})\Big)
    \\
    = &{p_{ik}^{n}}^\top\hspace{-2.5mm} + \rho_{ik}\Big(
    D_{ik}^\top(W_i^{n+1})\hspace{-1mm} + D_{ki}^\top(W_k^{n+1}) +
    D_{ki}^\top(W_k^{n} -W_k^{n+1} )\Big)
    \\
    = &{p_{ik}^{n}}^\top\hspace{-2.5mm} + \rho_{ik}\Big(
    G_{ik}^\top(W_i^{n+1}\hspace{-1mm},W_k^{n+1}) +
    D_{ki}^\top(W_k^{n} -W_k^{n+1} )\Big)
    \\
    = &{p_{ik}^{n+1}}^\top\hspace{-2.5mm} + \rho_{ik}
    D_{ki}^\top(W_k^{n}-W_k^{n+1}).
  \end{align*}
  Using this equation in~\eqref{eq:update_eq} and the fact that
  $G_{ik}(W_i^\star,W_k^\star)=0$, we obtain~\eqref{eq:prove_2}.
\end{proof}

\end{document}